\documentclass[11pt]{amsart} 
\usepackage{amssymb,amsmath,accents}
\usepackage{amscd}
\usepackage{amsfonts,amsthm,mathrsfs}
\usepackage{setspace}

\usepackage{graphicx, xcolor}

\definecolor{dGREEN}{rgb}{0.0,0.6,0.4}

\numberwithin{equation}{section}
\newtheorem{thm}{Theorem}[section]
\newtheorem{cor}[thm]{Corollary}
\newtheorem{lem}[thm]{Lemma}
\newtheorem{prop}[thm]{Proposition}

\newtheorem{remark}[thm]{Remark}

\makeatletter
\@namedef{subjclassname@2020}{\textup{2020} Mathematics Subject Classification}
\makeatother

\begin{document}

\title[On segmentation by total variation type energies]{On segmentation by total variation type energies of Kobayashi-Warren-Carter type with fidelity}

\author[Y.~Giga]{Yoshikazu Giga}
\address[Y.~Giga]{Graduate School of Mathematical Sciences, The University of Tokyo} 
\email{labgiga@ms.u-tokyo.ac.jp}

\author[A.~Kubo]{Ayato Kubo}
\address[A.~Kubo]{Department of Mathematics, Faculty of Science, Hokkaido University}
\email{kubo.ayato.j8@elms.hokudai.ac.jp}

\author[H.~Kuroda]{Hirotoshi Kuroda}
\address[H.~Kuroda]{Department of Mathematics, Faculty of Science, Hokkaido University}
\email{kuro@math.sci.hokudai.ac.jp}

\author[J.~Okamoto]{Jun Okamoto}
\address[J.~Okamoto]{Institute for the Advanced Study of Human Biology, Kyoto University}
\email{okamoto.jun.8n@kyoto-u.ac.jp}

\author[K.~Sakakibara]{Koya Sakakibara}
\address[K.~Sakakibara]{Faculty of Mathematics and Physics, Institute of Science and Engineering, Kanazawa University; RIKEN iTHEMS}\email{ksakaki@se.kanazawa-u.ac.jp}


\subjclass[2020]{49AQ20, 74A50, 74N05}
\keywords{total variation, fidelity, coincidence set, jump, Kobayashi-Warren-Carter energy}

\begin{abstract}
We consider a total variation type energy which measures the jump discontinuities different from usual total variation energy.
 Such a type of energy is obtained as a singular limit of the Kobayashi-Warren-Carter energy with minimization with respect to the order parameter.
 We consider the Rudin-Osher-Fatemi type energy by replacing relaxation term by this type of total variation energy.
 We show that all minimizers are piecewise constant if the data function in the fidelity term is continuous in one-dimensional setting.
 Moreover, the number of jumps is bounded by an explicit constant involving a constant related to the fidelity.
 This is quite different from conventional Rudin-Osher-Fatemi energy where a minimizer has no jumps if the data has no jumps.
 Our results give an upper bound of the number of segments in a segmentation problem.
 The existence of a minimizer is guaranteed in multi-dimensional setting when the data is bounded.
\end{abstract}

\maketitle
\thispagestyle{empty}

\section{Introduction} \label{SIn} 

We consider a kind of total variation energy which measures jumps different from the conventional total variation energy.
 Let $\Omega$ be a bounded domain in $\mathbb{R}^n$ with $n\geq1$.
 Let $u$ be in $BV(\Omega)$, i.e., its distributional derivative $Du$ is a finite Radon measure in $\Omega$ and let $|Du|$ denote its total variation measure.
 The total variation energy can be written in the form
\[
	TV(u) = \int_{\Omega\backslash J_u} |Du|
	+ \int_{J_u} |u^+ - u^-|\; d\mathcal{H}^{n-1},
\]
where $J_u$ denotes the (approximate) jump set of $u$ and $u^\pm$ is a trace of $u$ from each side of $J_u$; 
$\mathcal{H}^{n-1}$ denotes the $n-1$ dimensional Hausdorff measure.
 For a precise meaning of this formula, see Section \ref{SEx} and \cite{AFP}.
 Let $K(\rho)$ be a non-decreasing, lower semicontinuous function for $\rho\geq0$ with $K(0)=0$.
 We set
\[
	TV_K(u) = 	\int_{\Omega\backslash J_u} |Du|
	+ \int_{J_u} K \left(|u^+ - u^-|\right) d\mathcal{H}^{n-1}.
\]
For a given function $g\in L^2(\Omega)$, we are interested in a minimizer of
\[
	TV_{Kg}(u) = TV_K(u) + \mathcal{F}(u), \quad
	\mathcal{F}(u) = \frac{\lambda}{2} \int_\Omega |u-g|^2\; dx,
\]
where $\lambda>0$ is a constant.
 The term $\mathcal{F}$ is often called a fidelity term.
 If $TV_K=TV$, the functional $TV_g(u)=TV(u)+\mathcal{F}(u) $ is often called Rudin-Osher-Fatemi functional since its equivalent problem is proposed by \cite{ROF} to denoise the original image whose grey-level values equal $g$.
 (The equivalence between this problem and the original problem in \cite{ROF} is proved in \cite{CL97}.)
 For $TV_g$, there always exists a unique minimizer since the problem is strictly convex and lower semicontinuous in $L^2(\Omega)$.
 We are interested in regularity of a minimizer of $TV_{Kg}$ assuming some regularity of $g$.
 This problem is well studied for $TV_g$.
 Let $u_*$ be the minimizer of $TV_g$ for $g\in BV(\Omega)$.
 If $g$ is Lipschitz, then $u_*$ is Lipschitz and this Lipschitz regularity is optimal in the sense that $u_*$ may not be more regular even if $g$ is smooth.

Local and global Lipschitz regularity was first proved by \cite{CCN11} when $n\le7$.
 In the case of global Lipschitz  regularity, the convexity of $\Omega$ is assumed.
 It is also proved in \cite{CCN11} that $u_*$ is uniformly continuous if $g$ is uniformly continuous under the same restriction on $\Omega$ and $n$.
 Both local and global Lipschitz regularity were proved by \cite{P} without convexity of $\Omega$ nor dimension restriction by adjusting Bernstein's type estimate.
 By now it is well known that $J_{u_*}\subset J_g$ and $u_*^+(x) - u_*^-(x)\leq g^+(x)-g^-(x)$ for $\mathcal{H}^{n-1}$-a.e.\ $x\in J_{u_{\ast}}$.
 In particular, if $g$ has no jumps, so does $u_*$.
 This type of result was first obtained by \cite{CCN07} and extended to various settings in many years; see for example \cite{DS19, CJN13, Mer18}.
 In \cite{CL}, the above assertion for $TV_g$ is extended to vectorial case in a multi-dimensional domain even $TV$ is replaced by more general anisotropic one; see also a reviewer paper \cite{GKL}.

In this paper, like the Mumford-Shah functional \cite{MS89}, we shall show that a minimizer of $TV_{Kg}$ may have jumps even if $g$ has no jumps for some class of subadditive function $K$ including
\[
	K(\rho) = \frac{\rho}{1+\rho}
\]
as a particular example when $\Omega$ is an interval.
 This type of $TV_K$ appears as a kind of singular limit of the Kobayashi-Warren-Carter energy \cite{GOU}, \cite{GOSU}.
 Actually, we shall prove a stronger result saying that a minimizer is a piecewise constant function with finitely many jumps for $g\in C(\overline{\Omega})$ when $\Omega$ is a bounded interval.
 Here is a precise statement.

For a function  $K\colon[0,\infty)\rightarrow[0,\infty)$   measuring a jump, we assume that
%
\begin{enumerate}
\item[(K1)]  $K$  is a lower semicontinuous, non-decreasing function with $K(0)=0$.
\item[(K2)] For any $M>0$, there exists a positive constant $C_M$ such that 
\[
	K(\rho_1) + K(\rho_2) \geq K(\rho_1+\rho_2) + C_M \rho_1 \rho_2
\]
for all $\rho_1,\rho_2\geq0$ with $\rho_1+\rho_2\leq M$.
 In particular, $K$ is subadditive.
\item[(K3)] $\lim_{\rho\to0} K(\rho)/\rho=1$.
\end{enumerate}
If $K(\rho)=\rho$ so that $TV_K=TV$, $K$ satisfies (K1) and (K3) but does not satisfy (K2).
 If $K(\rho)=\rho/(1+\rho)$, $K$ satisfies (K2) as well as (K1) and (K3).
 Indeed, a direct calculation shows that
\[
	\frac{\rho_1}{1+\rho_1} + \frac{\rho_2}{1+\rho_2} 
	= \frac{\rho+2\rho_1\rho_2}{1+\rho+\rho_1\rho_2} = \frac{\rho}{1+\rho}
	+ \frac{(2+\rho)\rho_1\rho_2}{(1+\rho)(1+\rho+\rho_1\rho_2)}, 
	\quad \rho = \rho_1 + \rho_2.
\]
Thus, (K2) follows.
\begin{thm} \label{TMain}
Assume that $K$ satisfies (K1), (K2) and (K3).
 Assume that $g\in C[a,b]$.
 Let $M>0$ be a number such that
\[
	M \ge \operatorname{osc}_{[a,b]} 
	g := \max_{[a,b]} g - \min_{[a,b]} g.
\]
Let $U\in BV(a,b)$ be a minimizer of $TV_{Kg}$.
 Then $U$ must be a piecewise constant function (with finitely many jumps) satisfying $\inf g\leq U\leq \sup g$ on $[a,b]$.
 Let $m$ be the number of jumps of $U$.
 Then
\[
	m \leq \left[ (b-a)\lambda / A_M \right] +1,
\]
where $A_M=\min\{c_M/M, C_M\}$.
 Here, $c_M$ is a constant such that $K(\rho)\geq c_M\rho$ for $\rho\in[0,M]$ and $[r]$ denotes the integer part of $r\geq0$.
\end{thm}
We note that we do not assume that $g\in BV(a,b)$.
 If $g$ is non-decreasing or non-increasing, we have a sharper estimate for $m$.
\begin{thm} \label{TMainMon}
Assume that $K$ satisfies (K1), (K2) and (K3).
 Assume that $g\in C[a,b]$ is non-decreasing (resp.\ non-increasing).
 Let $M$ be taken such that $	M\geq\operatorname{osc}_{[a,b]}g=\left|g(b)-g(a)\right|$.
 Let $U\in BV(a,b)$ be a minimizer of $TV_{Kg}$.
 Then $U$ must 
be a non-decreasing (non-increasing) piecewise constant function satisfying $\inf g\leq U\leq\sup g$ on $[a,b]$.
 The number $m$ of jumps is estimated as
\[
	m \leq \left[(b-a)\lambda/C_M \right] + 1.
\]
\end{thm}

To show Theorem \ref{TMain} and also Theorem \ref{TMainMon}, we introduce the notion of a coincidence set $C$, which is formally defined by
\[
	C = \left\{ x \in [a,b] \bigm|
	U(x) = g(x) \right\}.
\]
It turns out that a minimizer $U$ of $TV_{Kg}$ is always continuous on $C$ and outside this set $C$, $U$ is piecewise constant and
\begin{equation} \label{EB0}
	\sup_{[a,b]}U \leq \sup_{[a,b]}g, \quad
	\inf_{[a,b]}U\geq\inf_{[a,b]}g.
\end{equation}
Moreover, we are able to prove that $U$ has at most one jump on $(\alpha,\beta)$ if $(\alpha,\beta)\cap C=\emptyset$ and $\alpha,\beta\in C$.
 We also prove that if $\alpha,\beta\in C$ with $\alpha<\beta$ too close, then $U$ must be monotone on $(\alpha,\beta)$.
 Here, $(\alpha,\beta)$ may include some point of $C$.
 To show these properties, it suffices to assume the subadditivity of $K$ instead of (K2).
 It includes the case $TV_K$, where $K(\rho)=\rho$ so that $TV_K$ is the standard total variation.

We shall prove that if $\alpha\in C$ and $\beta\in C$ with $\alpha<\beta$ is too close, $U$ must be a constant on $(\alpha,\beta)$.
 A key step (Lemma \ref{LLEN}) is to show that for a non-decreasing minimizer $U$ on $(a,b)$ and $\alpha,\beta\in C$ with $\alpha<\beta$, the estimate
\[
	\beta - \alpha > C_M/\lambda
\]
holds provided that there is $\gamma\in C\cap(\alpha,\beta)$ and $(\gamma,\beta)\cap C=\emptyset$ and that $\rho_2:=U(\beta)-U(\gamma)\geq\rho_1:= U(\gamma)-U(\alpha)$, where $C_M$ is the constant in (K2).
 Here is a rough strategy to prove the statement.
 Since $(\gamma,\beta)\cap C=\emptyset$, there is exactly one  jump point $x_1$ in $(\gamma,\beta)$.
 We compare $TV_{Kg}(U)$ and $TV_{Kg}(v)$ in $(\alpha,\beta)$, where $v$ equals $U(\alpha)$ on $(\alpha,x_1)$ and equals $U$ on $(x_1,\beta)$ (see Figure \ref{FGUV}).
 It is not difficult to prove
\begin{equation} \label{EB1}
	TV_K(U) - TV_K(v) \geq K(\rho_1) + K(\rho_2) - K(\rho_1+\rho_2)
	\geq C_M \rho_1 \rho_2
\end{equation}
since $\rho_1+\rho_2\leq M$ by \eqref{EB0}.
 The proof for 
\begin{equation} \label{EB2}
	\frac2\lambda \left(\mathcal{F}(U) - \mathcal{F}(v) \right)
	\geq - 2\rho_1 \rho_2 (x_1-\alpha)
\end{equation}
is more involved.
 It is not difficult to show
\begin{equation} \label{EB3}
	\int_\gamma^{x_1} \left\{ (U-g)^2 - (v-g)^2 \right\} dx
	\geq - \rho_1(\rho_1+\rho_2)(x_1-\gamma).
\end{equation}
The proof for
\begin{equation} \label{EB4}
	\int_\alpha^\gamma \left\{ (U-g)^2 - (v-g)^2 \right\} dx
	\geq - \rho_1^2(\gamma-\alpha)
\end{equation}
is more difficult.
 If $g$ is non-decreasing, we are able to prove
\[
	g(x) - U(x) \leq U(x_2+0) - U(x_2-0)
	\quad\text{for}\quad  x \in F,
\]
where $F=[x_0,x_2]\subset[a,b)$ is a maximal closed interval (called a facet) such that $U$ is a constant in the interior $\operatorname{int}F$ of $F$.
 For general interval $[x_0,x_2]$, one only expects such an inequality just in the average sense, i.e.,
\begin{equation} \label{EB5}
	\int_{x_0}^{x_2} \left( g(x) - U(x) \right) dx
	\leq \left( U(x_2+0) - U(x_2-0) \right) (x_2-x_0).
\end{equation}

The estimate \eqref{EB4} 
follows from \eqref{EB5}.
 Combining \eqref{EB3} and \eqref{EB4}, we obtain \eqref{EB2} since $\rho_2\geq\rho_1$.
 The estimates \eqref{EB1} and \eqref{EB2} yield
\[
	TV_{Kg}(U) - TV_{Kg}(v) \geq \rho_1 \rho_2 \left(C_M - (x_1-\alpha)\lambda \right).
\]
If $\beta-\alpha\le C_M/\lambda$, $U$ cannot be a minimizer.
 Similarly, we are able to prove that if $U$ is continuous on $(\alpha,\beta)$ with $\alpha,\beta\in C$, then $U$ is a constant on $(\alpha,\beta)$.
 These observations yield Theorem \ref{TMain}.
 Theorem \ref{TMainMon} can be proved similarly to Theorem \ref{TMain} by noting that a minimizer $U$ is a monotone function.

In a forthcoming paper, we give a quantitative estimate for $TV_{Kg}$ for monotone data $g\in C[a,b]$.
 It is easy to estimate $TV_{Kg}(u)$ for a piecewise constant $u$ from below.
 We approximate a general $BV$ function $u$ by piecewise constant functions $u_m$ so that $TV_{Kg}(u_m)\to TV_{Kg}(u)$.
 Using such an approximation result, we establish an estimate of $TV_{Kg}$ for a general $BV$ function.
 We notice that such an estimate gives another way to prove Theorem \ref{TMainMon} with improvement of the estimate of $m$ by $m\le\left[(b-a)\lambda\bigm/(2C_M)\right]+1$.

It is not difficult to prove that $TV_{Kg}$ is lower semicontinuous in the space of piecewise constant functions with at most $k$ jumps.
 Since the space is of finite dimension, its bounded closed set is compact.
 By Weierstrass' theorem, $TV_{Kg}$ admits a minimizer $v$ among piecewise constant functions with at most $k$ jumps with $\inf g\leq v\leq\sup g$.
 Thus, Theorem \ref{TMain} guarantees the existence of a minimizer in $BV(a,b)$.
 The existence of minimizer of $TV_{Kg}$ itself can be proved for general essential bounded measurable function $g$, i.e., $g\in L^\infty(\Omega)$ for general Lipschitz domain in $\mathbb{R}^n$ since it is known \cite{AFP} that $TV_K$ is lower semicontinuous in a suitable topology.
 We shall discuss this point in Section \ref{SEx}.
 Note that instead of (K2) subadditivity for $K$ is enough to have the existence of a minimizer.

We believe Theorem \ref{TMain} extends for general $g\in L^\infty(a,b)$.
 In a forthcoming paper, we prove a weaker version asserting that there exists a piecewise constant minimizer satisfying the same property as in Theorem \ref{TMain}.
 All piecewise constant minimizers must have the same property but it does not exclude the possibility that there exists a non-piecewise constant minimizer.

Our results give an upper bound of the number of segments in a segmentation problem.
 The problem is finds $N\in\mathbb{N}$ and partition $\{I_1,\ldots,I_N\}$ and associate values $\{U_1,\ldots,U_N\}$ that minimizes
\[
	\sum_{j=1}^{N-1} K \left( |U_{j+1} - U_j| \right)
	+ \sum_{j=1}^N \frac\lambda2 \int_{I_j} \left|U_j - g(x) \right|^2\,dx.
\]
Theorem \ref{TMain} yields
\begin{cor} \label{CInt}
$N=m+1\le\left[(b-a)\lambda\bigm/A_M\right]+2$.
\end{cor}
The number of clusters or segments is bounded from above by a constant independent of detailed behavior of $g$.

The case $K(\rho)\equiv1$ for $\rho>0$ ($K(0)=0$) is sometimes called Pott's model originally introduced by Mumford and Shah \cite{MS89}.
 This $K\equiv1$ of course violates (K3).
 However, it turns out that this $K$ satisfies (K1), (K2) and a weaker assumption
\begin{enumerate}
\item[(K3w)] For $M>0$, there exists a positive constant $c_M$ such that
\[
	K(\rho) > c_M\rho
	\quad\text{for all}\quad
	0 < \rho \le M
\]
\end{enumerate}
and our estimate for $m$ is still valid by restricting the space of functions on the space of piecewise constant functions.
 For higher dimensional segmentation problem, the reader is referred to \cite{CT91}, \cite{MT93}.
 For example, in \cite{CT91}, the existence of a minimizer $U$ was proved.
 Moreover, in \cite{CT91} it is proved that $J_U$ is a closed set $K$ up to $\mathcal{H}^{n-1}$ measure zero set and that $U\in C^1(\Omega\backslash K)$ with $\nabla U=0$ in $\Omega\backslash K$.
 The paper \cite{Cha92} focuses one-dimensional problem.

Let us state our result for piecewise constant minimizers.

\begin{thm} \label{TPie}
	The conclusions of Theorem \ref{TMain}, Theorem \ref{TMainMon} and Corollary \ref{CInt} are still valid if we replace (K3) by (K3w) provided that $U$ is assumed to be piecewise constant (possibly with infinitely many jumps), i.e., $\int_{\Omega\backslash J_U} |DU|=0$ for $\Omega=(a,b)$.
\end{thm}
Our theorems also say that our minimizer is piecewise constant if $g$ is piecewise constant and $J_U\subset J_g$ since $U$ must be constant on $(\alpha,\beta)$ if $g$ is constant there.
 This type of preservation of piecewise constant structure for $TV_g$ is found in  \cite[2.4.2]{BF12} in one-dimensional setting.
 It is extended to vector-valued case for $TV_g$ and its anisotropic variants in \cite{GL19, GL24} by proving $|Du|\le|Dg|$.
 For the evolution problem, we also note that there is a characterization of faceted part of the minimizer via dual problem $\int|\operatorname{div}u|+\mathcal{F}$ for general $g\in L^2$; see \cite[Proposition 4.2]{BCNO11}.
 For higher-dimensional problem for $TV_g$, the preservation of piecewise constant structure is studied in \cite{BCN05, LMM17, KSS19}.
 Contrary to there results, our result says that minimizer in $BV$ is always piecewise constant for any $g$ ($\in C[a,b]$) with no jumps.

A natural open problem for $TV_{Kg}$ is for vector-valued functions.
 It is also natural to ask what happens when $\Omega$ is multi-dimensional.
 Here piecewise constant should be associated with Cacciopolli partition \cite{AFP}.
 However, it is not clear what is the number of jumps in multi-dimensional case.

As shown in \cite{GOU}, \cite{GOSU}, $TV_{Kg}$ is obtained as a singular limit ($\varepsilon\downarrow0$ limit) of the Kobayashi-Warren-Carter type energy  \cite{KWC1,KWC2,WKC}
\begin{align*}
	E_{\mathrm{KWC}g}^\varepsilon (u,v) &:= E_{\mathrm{KWC}}^\varepsilon (u,v) + \mathcal{F}(u) \\
	E_{\mathrm{KWC}}^\varepsilon (u,v) &:= \int_\Omega sv^2 |Du| + E_{\mathrm{sMM}}^\varepsilon (v), \quad s>0 \\
	E_{\mathrm{sMM}}^\varepsilon (v) &:= \frac{\varepsilon}{2} \int_\Omega |\nabla v|^2\; dx
	+ \frac{1}{2\varepsilon} \int_\Omega F(v)\; dx
\end{align*}
by minimizing order parameter $v$; here $F(v)$ is a single-well potential typically of the form $F(v)=(v-1)^2$ and $s$ is a positive parameter.
 In fact, in one-dimensional setting \cite{GOU}, the Gamma limit of $E_{\mathrm{sMM}}^\varepsilon (v)$ under the graph convergence formally equals
\[
	E_{\mathrm{sMM}}^0 (\Xi)
	= \sum_{i=1}^\infty 2 \left( G(\xi_i^-) + G(\xi_i^+) \right), \quad
	G(\tau) = \left| 
	\int_1^\tau \sqrt{F(\xi)}\; d\xi \right|
\]
if the limit of $v$ in $\Omega=(a,b)$ equals a set-valued function $\Xi$ of the form
\[
	\Xi(x) = \left\{
\begin{array}{ll}
	&\hspace{-0.7em}1, \quad x \notin \Sigma, \\
	&\hspace{-0.7em}[\xi_i^-, \xi_i^+] (\ni 1) \quad\text{for}\quad x_i \in \Sigma,
\end{array}
\right.
\]
where $\Sigma$ is some (at most) countable set.
 The Gamma limit of $E_{\mathrm{KWC}}^\varepsilon$ 
equals
\[
	E_{\mathrm{KWC}}^0(u,\Xi) = \sum_{i=1}^\infty \left( s(\xi_{i,+}^-)^2 \left|u^+(x_i)-u^-(x_i)\right|
	+ 2 \left( G(\xi_i^-) + G(\xi_i^+)\right) \right)
	+ \int_{\Omega\backslash J_u} s|Du|,
\]
where $\xi_+=\max(\xi,0)$ and $x_i\in\Sigma$.
 Here for $u$, $L^1$ type limit is considered.
 If we minimize $E_{\mathrm{KWC}}^0(u,\Xi)$ with fixed $u$, $\xi^+$ must be one since $[\xi^-,\xi^+]\ni1$ and $G(1)=0$.
 Thus
\begin{align*}
	\inf_\Xi E_{\mathrm{KWC}}^0(u,\Xi)
	&= \sum_{i=1}^\infty \min_{\xi>0} \left( s(\xi_+)^2 \left|u^+(x_i) - u^-(x_i)\right| + 2G(\xi) \right)
	+ \int_{\Omega\backslash J_u} s |Du| \\
	&= sTV_K(u),
\end{align*}
where
\begin{equation} \label{EKdef}
	K(\rho) = \min_\xi \left( (\xi_+)^2 \rho + 2G(\xi)/s \right).
\end{equation}
In the case, $s=1$ and $F(v)=(v-1)^2$, a direct calculation shows that 
\[
	K(\rho) = \min_{\xi>0} \left( \xi^2 \rho + (\xi-1)^2 \right)
	= \frac{\rho}{\rho+1}.
\]
In Section \ref{SKdef}, we shall prove that $F\in C^1(\mathbb{R})$ satisfies $K$ defined in \eqref{EKdef} satisfies (K2) provided that
\[
	\varlimsup_{v\uparrow1} F'(v) \bigm/ (v-1) < \infty
\]
and that $F\ge0$ and $F(v)=0$ if and only if $v=0$.
 It turns out that $K$ obtained by \eqref{EKdef} is bounded while $K$ satisfying (K1), (K2), (K3) may not be bounded.
 A typical example is $K(\rho)=\rho/(1+\rho)^{1/2}$, which satisfies (K1), (K2), (K3).

If we replace $\int sv^2|Du|$ by $\int sv^2|Du|^2$, the energy corresponding to $E_{\mathrm{KWC}g}^\varepsilon$ is nothing but what is called the Ambrosio-Tortorelli energy \cite{AT}.
 Its singular limit is a Mumford-Shah functional
\[
	E_\mathrm{MS} (u,K) := s \int_{\Omega\backslash K} |\nabla u|^2 + \mathcal{H}^{n-1}(K) + \mathcal{F}(u),
\]
where $K$ is a closed set in $\Omega$ \cite{AT, AT2, FL}.
 The existence of a minimizer is obtained in \cite{GCL} by using the space of special functions with bounded variation, i.e., $SBV$ functions which is a subspace of $BV(\Omega)$.

A modified total variation energy $TV_K$ is not limited to a singular limit of the Kobayashi-Warren-Carter energy.
 In fact, $TV_K(u)$ like energy is derived as the surface tension of grain boundaries in polycrystals by \cite{LL}, where $u$ is taken as a piecewise constant (vector-valued) function; see also \cite{FGaSp} for more recent development.
 The function $K$ measuring jumps may not be isotropic but still concave.
 In \cite{ELM}, $TV_K$ type energy for a piecewise constant function is also considered to study motion of a  grain boundary.
 However, in their analysis, the convexity of $K$ is assumed.

This paper is organized as follows.
 In Section \ref{SEx}, we give a rigorous formulation of $TV_K$ and prove the existence of a minimizer of $TV_{Kg}$ for $g\in L^\infty(\Omega)$ for a general Lipschitz domain in $\mathbb{R}^n$.
 In Section \ref{SCoi}, for $n=1$ we study a profile of a minimizer $U$ for $TV_{Kg}$ including $TV_g$ outside the coincidence set.
 We also prove that $U$ is monotone in $(\alpha,\beta)$ with $\alpha,\beta\in C$ provided that $\alpha$ and $\beta$ is close.
 In Section \ref{SGe}, we prove that $\alpha,\beta\in C$ cannot be too close if $K$ satisfies (K2).
 We prove Theorem \ref{TMain} and Theorem \ref{TMainMon}.
 In Section \ref{SKdef}, we shall discuss a sufficient condition that $K$ in \eqref{EKdef} satisfies (K2).


\section{Existence of a minimizer} \label{SEx} 
In this section, after giving a precise definition of $TV_K$, we give an existence result for its Rudin-Osher-Fatemi type energy.
 The proof is based on a standard compactness result for $TV$ and a classical lower semicontinuity result for $TV_K$.

We recall a standard notation as in \cite{AFP}.
 Let $\Omega$ be a domain in $\mathbb{R}^n$.
 For a locally integrable (real-valued) function $u$, its total variation $TV(u)$ in $\Omega$  is defined as 
\[
	TV(u) = \sup \left\{ \int_\Omega - u \operatorname{div} \varphi\; dx \biggm|
	\varphi \in C_c^\infty(\Omega,\mathbb{R}^n),\ 
	\left| \varphi(x) \right| \leq 1\ \text{in}\ \Omega \right\},
\]
where $C_c^\infty(\Omega,\mathbb{R}^n)$ denotes the space of all $\mathbb{R}^n$-valued smooth functions with compact support in $\Omega$.
 An integrable function $u$, i.e., $u\in L^1(\Omega)$, is called a function of bounded variation if $TV(u)<\infty$.
 The space of all such function is denoted by $BV(\Omega)$, i.e.,
\[
	BV(\Omega) = \left\{ u \in L^1(\Omega) \bigm|
	TV(u) < \infty \right\}.
\]
By Riesz's representation theory, one easily observe that $TV(u)$ is finite if and only if the distributional derivative $Du$ of $u$ is a finite Radon measure on $\Omega$ and its total variation $|Du|(\Omega)$ in $\Omega$ equals $TV(u)$.

We next define a jump discontinuity of a locally integrable function.
 Let $B_r(x)$ denote an open ball of radius $r$ centered at $x$ in $\mathbb{R}^n$.
 In other words, 
\[
	B_r(x) = \left\{ y \in \mathbb{R}^n \bigm|
	| y - x | < r \right\}.
\]
For a unit vector $\nu\in\mathbb{R}^n$, we define a half ball of the form
\[
	B_r^\pm(x,\nu) = \left\{ y \in B_r(x) \bigm|
	\pm\nu \cdot ( y - x ) \geq 0 \right\},
\]
where $a\cdot b$ for $a,b\in\mathbb{R}^n$ denotes the standard inner product in $\mathbb{R}^n$.
 Let $w$ be a locally integrable function in $\Omega$.
 We say that a point $x\in\Omega$ is a (approximate) jump point of $w$ if there exists a unit vector $\nu_w\in\mathbb{R}^n$, $w^\pm\in\mathbb{R}$, $w^+\neq w^-$, such that
\[
	\lim_{r\downarrow0} \frac{1}{\mathcal{L}^n\left(B_r^\pm(x,\nu_w)\right)}
	\int_{B_r^\pm(x,\nu_w)} \left| w(y) - w^\pm \right| dy = 0.
\]
Here $\mathcal{L}^n$ denotes the Lebesgue measure in $\mathbb{R}^n$ so this integral is the average of $\left|w(y)-w^\pm\right|$ over $B_r^\pm(x,\nu_w)$.
 The set of all jump points of $w$ is denoted by $J_w$ and called the (approximate) jump (set) of $w$.
 By definition, $J_w$ is contained in the set $S_w$ of (approximate) discontinuity point of $w$, i.e.,
\begin{multline*}
	S_w = \Biggl\{ x \in \Omega \biggm|
	\lim_{r\downarrow0} \frac{1}{\mathcal{L}^n\left(B_r(x)\right)}
	\int_{B_r(x)} \left| w(y) - z \right| dy = 0\ \text{does not hold} \\
	\text{for any choice of}\ z \in \mathbb{R} \Biggr\}.
\end{multline*}

By the Federer-Vol'pert theorem \cite[Theorem 3.78]{AFP}, we know $\mathcal{H}^{n-1}(S_u\backslash J_u)=0$ provided that $u\in BV(\Omega)$.
 Moreover, $S_u$ is countably $\mathcal{H}^{n-1}$-rectifiable, i.e., $S_u$ can be covered by a countable union of Lipschitz graphs up to an $\mathcal{H}^{n-1}$ measure zero set.
 In particular, $J_u$ is also countably $\mathcal{H}^{n-1}$-rectifiable.
 Quite recently, it is proved that $J_u$ is always countably $\mathcal{H}^{n-1}$-rectifiable if we only assume that $w$ is locally integrable \cite{DN}.
 For $u\in BV(\Omega)$, the value $u^\pm$ can be viewed as a trace of $u$ on a countably $\mathcal{H}^{n-1}$-rectifiable set \cite[Theorem 3.77, Remark 3.79]{AFP} except $\mathcal{H}^{n-1}$ negligible set (up to permutation of $u^+$ and $u^-$).
 We now recall a unique decomposition of the Radon measure $Du$ for $u\in BV(\Omega)$ of the form
\[
	Du = D^a u + D^s u, \quad
	D^s u = D^c u + (u^+ - u^-) \cdot \nu_u \mathcal{H}^{n-1} \lfloor J_u;
\]
 see \cite[Section 3.8]{AFP}.
 The term $D^au$ denotes the absolutely continuous part and $D^su$ denotes the singular part with respect to the Lebesgue measure.
 The term $D^au=\nabla u\mathcal{L}^n$, where $\nabla u\in\left(L^1(\Omega)\right)^n$.
 The singular part is decomposed into two parts; $D^cu$ vanishes on sets of finite $\mathcal{H}^{n-1}$ measure.
 For a measure $\mu$ on $\Omega$ and a set $A\subset\Omega$, the associate measure $\mu\lfloor A$ is defined as
\[
	(\mu \lfloor A)(W) = \mu(A \cap W), \quad
	W \subset \Omega.
\]

We now consider a total variation type energy measuring jumps in a different way.
 For $u\in BV(\Omega)$, we set
\[
	TV_K(u) := \left( TV_K(u,\Omega) := \right) \int_{\Omega\backslash J_u} |Du|
	+ \int_{J_u} K \left(|u^+ - u^-|\right) d\mathcal{H}^{n-1}.
\]
Here the density function is assumed to satisfy following conditions.
\begin{enumerate}
\item[(K1)] $K:[0,\infty)\to[0,\infty)$ is lower semicontinuous, non-decreasing with $K(0)=0$.
\item[(K2w)] $K$ is subadditive, i.e., $K(\rho_1+\rho_2)\leq K(\rho_1)+K(\rho_2)$.
\item[(K3)] $\lim_{\rho\to0} K(\rho)/\rho=1$.
\end{enumerate}
\begin{thm} \label{TEx1}
Assume that $K$ satisfies (K1), (K2w) and (K3).
 Let $\Omega$ be a bounded domain with Lipschitz boundary in $\mathbb{R}^n$.
 Let $\mathcal{E}$ be a lower semicontinuous function in $L^1(\Omega)$ with values in $[0,\infty]$.
 Then $TV_K+\mathcal{E}$ 
has a  minimizer on $BV(\Omega)$ provided that a coercivity condition
\[
	\inf_{u\in BV(\Omega)} (TV_K + \mathcal{E})(u)
	= \inf_{\substack{u\in BV(\Omega) \\ \lVert u\rVert_\infty\leq M}} (TV_K + \mathcal{E})(u)
\]
for some $M>0$, where $\lVert\cdot\rVert_\infty$ denotes the $L^\infty$-norm.
\end{thm}

We consider the Rudin-Osher-Fatemi type energy for $TV_K$, i.e.,  for $g\in L^2(\Omega)$, 
\[
	TV_{Kg} (u) := TV_K(u) + \mathcal{F}(u), \quad
	\mathcal{F}(u) = \frac{\lambda}{2} \int_\Omega |u-g|^2\; dx.
\]
\begin{thm} \label{TEx2}
Assume that $K$ satisfies (K1), (K2w) and (K3).
 Let $\Omega$ be a bounded domain with Lipschitz boundary in $\mathbb{R}^n$.
 Assume that $g\in L^\infty(\Omega)$.
 Then there is an element $u_0\in BV(\Omega)$ such that
\[
	TV_{Kg} (u_0) = \inf_{u\in BV(\Omega)} TV_{Kg}(u).
\]
In other words, there is at least one minimizer of $TV_{Kg}$.
\end{thm}
\begin{proof}[Proof of Theorem \ref{TEx2} admitting Theorem \ref{TEx1}.]
In the case $\mathcal{E}=\mathcal{F}$, the lower semicontinuity of $\mathcal{E}$ in $L^1(\Omega)$ is rather clear.
 If $g\in L^\infty(\Omega)$, then for a chopped function  $u_M=\max\left(\min(u,M),-M\right)$ with $M>\lVert g\rVert_\infty$, we see that
\begin{align*}
	&\frac2\lambda \left(\mathcal{F}(u) - \mathcal{F}(u_M) \right)
	= \int_{u\geq M} |u-g|^2\; dx - \int_{u\geq M} |M-g|^2\; dx \\
	&+ \int_{u\leq-M} |u-g|^2\; dx - \int_{u\leq-M} |{-}M-g|^2\; dx \geq 0, \\
	&TV_K(u) - TV_K(u_M) \ge 0.
\end{align*}
Thus the coercivity condition is fulfilled, and  Theorem \ref{TEx2} follows from Theorem \ref{TEx1}.
\end{proof}
%
In the rest of this section, we shall prove Theorem \ref{TEx1} by a simple direct method.
 We begin with compactness.
\begin{prop} \label{PCom}
Assume that (K1) and (K3) are fulfilled.
 Assume that  $\Omega$ is a bounded domain with Lipschitz boundary in $\mathbb{R}^n$.
 Let $\{u_k\}_{k=1}^\infty$ be a sequence in $BV(\Omega)$ such that
\[
	\sup_{k\geq1} TV_K(u_k) < \infty \quad\text{and}\quad
	\sup_{k\geq1} \lVert u_k \rVert_\infty < \infty.
\]
Then there is a subsequence $\{u_{k'}\}$ and $u\in BV(\Omega)$ such that $u_{k'}\to u$ strongly in $L^1(\Omega)$ and $Du_{k'}\to Du$ weak* in the space of bounded measures.
 In other words, $u_{k'}$ sequentially weakly* converges to $u$ in $BV(\Omega)$.
\end{prop}
\begin{proof}
By (K1) and (K3), we see that for any $M$, there is $c_M \in(0,1)$ such that
\begin{equation} \label{EKEB}
	K(\rho) > c_M\rho \quad\text{for}\quad
	0 < \rho \leq M.
\end{equation}
In other words, (K3w) holds.
 If $M$ is chosen such that $\lVert u_k \rVert_\infty\leq M$, then
\[
	TV(u_k) \leq \frac{1}{c_M} TV_K(u_k).
\]
Thus $\left\{TV(u_k)\right\}$ is bounded.
 By the standard compactness for $BV(\Omega)$ function \cite[Theorem 3.23]{AFP}, \cite[Theorem 1.19]{Giu} yields the desired results.
\end{proof}

For a lower semicontinuity, we have
\begin{prop} \label{PLS}
Assume (K1), (K2w) and (K3), then $TV_K$ is sequentially weakly* lower semicontinuous in $BV(\Omega)$.
\end{prop}

This is a special form of the lower semicontinuity result \cite[Theorem 5.4]{AFP}.
 We restate it for the reader's convenience. 
 We consider
\[
	F(u) = \int_{\Omega\backslash J_u} \varphi \left( |\nabla u| \right) dx
	+ \beta |D^c u|(\Omega)
	+ \int_{J_u} K \left( | u^+ - u^- | \right) d\mathcal{H}^{n-1}.
\]
\begin{prop} \label{PLS2}
Let $\varphi:[0,\infty)\to[0,\infty)$ be a non-decreasing, lower semicontinuous and convex function.
 Assume that $K:(0,\infty)\to[0,\infty)$ is a non-decreasing, lower semicontinuous and subadditive function and $\beta\in[0,\infty)$.
 Then $F$ is sequentially weakly* lower semicontinuous in $BV(\Omega)$ provided that
\[
	\lim_{t\uparrow\infty} \frac{\varphi(t)}{t} = \beta
	= \lim_{t\downarrow0} \frac{K(t)}{t}.
\]
\end{prop}
See \cite[Theorem 5.4]{AFP}.
 Such a lower semicontinuity is originally due to Bouchitt\'e and Buttazzo \cite{BB}.
 In our setting $\varphi(t)=t$, $\beta=1$.
\begin{proof}[Proof of Theorem \ref{TEx1}]
Let $\{u_k\}_{k=1}^\infty$ be a minimizing sequence of $TV_K+\mathcal{E}$, i.e.,
\[
	\lim_{k\to\infty} (TV_K + \mathcal{E})(u_k)
	= \inf_{u\in BV} (TV_K + \mathcal{E}),
\]
which is bounded in $L^\infty(\Omega)$.
 By compactness (Proposition \ref{PCom}), $\{u_k\}$ contains a convergent subsequence still denoted by $\{u_k\}$ to some $u\in BV(\Omega)$, sequentially strong in $L^1(\Omega)$ and weakly* in $BV(\Omega)$.
 By lower semicontinuity of $TV_K$ (Proposition \ref{PLS}), we conclude that
\[
	(TV_K + \mathcal{E})(u) \leq \varliminf_{k\to\infty}(TV_K + \mathcal{E})(u_k).
\]
Thus, $u$ is a minimizer of $TV_K+\mathcal{E}$ in $BV(\Omega)$.
\end{proof}


\section{Coincidence set of a minimizer} \label{SCoi} 

In this section, we discuss one-dimensional setting and study properties of coincidence set
\[
	C = \left\{ x \in \Omega \bigm|
	U(x) = g(x) \right\}
\]
of a minimizer $U$.

Let $\Omega$ be a bounded open interval, i.e., $\Omega=(a,b)$.
 We consider $TV_{Kg}(u)$ for $g\in C[a,b]$ for $u\in BV(a,b)$.
 Since $u$ can be written as a difference of two non-decreasing function, $J_u$ is at most a countable set and outside $J_u$, $u$ is continuous.
 Moreover, we may assume that $u$ is right (resp.\ left) continuous at $a$ (resp.\ at $b$).
 For $x\in J_u$, $u(x\pm0)$ is well-defined.

For $K$, we assume (K1).
 We first prove that a minimizer $U$ is piecewise constant in the place $U>g$ or $U<g$.
\begin{lem} \label{LPic}
Assume that $K$ satisfies (K1) and that $g\in C[a,b]$.
 Let $U\in BV(a,b)$ be a minimizer of $TV_{Kg}$.
 Let $x_0\in(a,b)$ be a continuous point of $U$ and assume $U(x_0)>g(x_0)$ (resp.\ $U(x_0)<g(x_0)$).
 Then $U$ is constant in some interval $(\alpha,\beta)$ including $x_0$ and $U(x)>g(x)$ (resp.\ $U(x)<g(x)$) for $x\in(\alpha,\beta)$.
 Moreover, we can take $(\alpha,\beta)$ such that 
one of following three cases occurs exclusively.
\begin{enumerate}
\item[(i)] $U(\alpha-0)<g(\alpha)<U(\alpha+0)$ (resp.\ $U(\alpha-0)>g(\alpha)>U(\alpha+0)$) and $U(\beta-0)=g(\beta)$,
\item[(i\hspace{-1pt}i)] $U(\alpha+0)=g(\alpha)$ and $U(\beta-0)>g(\beta)>U(\beta+0)$ (resp.\ $U(\beta-0)<g(\beta)<U(\beta+0)$),
\item[(i\hspace{-1pt}i\hspace{-1pt}i)] $U(\alpha+0)=g(\alpha)$ and $U(\beta-0)=g(\beta)$.
\end{enumerate}
In the case $\alpha=a$, we do not impose the condition $g(\alpha)>U(\alpha-0)$ (resp.\ $g(\alpha)<U(\alpha-0)$) and similarly, $g(\beta)>U(\beta+0)$ (resp.\ $g(\beta)<U(\beta+0)$) is not imposed for $\beta=b$ since $U(\alpha-0)$, $U(\beta+0)$ are undefined.
\end{lem}
\begin{proof}
We shall only give a proof for the case $U(x_0)>g(x_0)$ since the argument for $U(x_0)<g(x_0)$ is symmetric.
 We consider the case that $U$ is continuous on $(\alpha,\beta)$ with $U>g$ on $(\alpha,\beta)$ and $U(\alpha+0)=g(\alpha)$, $U(\beta-0)=g(\beta)$.
 We shall claim that $U$ is a constant.
 If not, $\max_{[\alpha,\beta]}U>\min_{[\alpha,\beta]}U$.
 There would exist two points $x_1$ and $x_2$ in $[\alpha,\beta]$ such that $U(x_1)=\min_{[\alpha,\beta]}U$ and $U(x_2)=\max_{[\alpha,\beta]}U$ such that $U(x)\in(\min U,\max U)$ for $x$ between $x_1$ and $x_2$.
 We may assume $x_1<x_2$ with $U(x_1)<U(x_2)$ since the proof for the other case is symmetric.
 We first observe that $U$ must be a non-decreasing function on $[x_1,x_2]$.
 If not, there is either a local minimum or a local maximum of $U$ in $(x_1, x_2)$.
 In both cases, the argument is symmetric, we only discuss the case when there is a local maximum.
 In this case, there is $\gamma\in(x_1,x_2)$ and an open interval $I_\gamma=(y_1,y_2)\subset(x_1,x_2)$ containing $\gamma$ such that $\max_{\bar{I}_\gamma}U=U(\gamma)$ and $\max_{\partial I_\gamma}U<U(\gamma)$ with $U(y_1)=U(y_2)$.
 We set
\[
	v(x) := \left\{
\begin{array}{l}
	\max \left\{ U(y_1), U(x) - \varepsilon \right\}, \quad x \in I_\gamma \\
	U(x), \quad x \notin I_\gamma,
\end{array}
\right.
\]
where $\varepsilon$ is taken so that $U(x)-\varepsilon>\max\left(\max_{\partial I_\gamma}g, g(x)\right)$ for $x\in I_\gamma$.
 Since, $U>v>g$ on $I_\gamma$, we see that $\mathcal{F}(v)<\mathcal{F}(U)$.
 Evidently, $TV_K(v)=TV(v)\leq TV(U)=TV_K(U)$.
 This would contradict to the assumption that $U$ is a minimizer of $TV_{Kg}$.
 Hence, $U$ is monotone in $[x_1,x_2]$ satisfying $U>g$.
 Suppose that $U$ were not a constant function, there would exist another monotone continuous function $\tilde{U}$ ($\le U$) on $[x_1,x_2]$ such that $\tilde{U}(x_1)=U(x_1)$, $\tilde{U}(x_2)=U(x_2)$ keeping the property that $\tilde{U}>g$ on $(x_1,x_2)$.
 We set
\[
	v(x) = \left\{
\begin{array}{ll}
	\tilde{U}(x), & x \in (x_1, x_2), \\
	U(x), & x \notin (x_1, x_2).
\end{array}
\right.
\]
Since $U>v>g$ on $(x_1,x_2)$, we see that $\mathcal{F}(v)<\mathcal{F}(U)$.
 Clearly, $TV_K(v)=TV_K(U)$ so $U$ would not be a minimizer,
 This is a contradiction so we conclude that $U$ is a constant funtion on $[x_1,x_2]$ and (i\hspace{-1pt}i\hspace{-1pt}i) occurs.
 The proof shows that $U$ is constant in any interval $[x_1,x_2]$ provided that $U$ is continuous on $[x_1,x_2]$ and $U>g$ on $(x_1,x_2)$.

Assume that there is a jump point $\alpha_0$ of $U$ with $\alpha_0<x_0$ with $U(\alpha_0+0)>g(\alpha_0)$.
 Then $U(\alpha_0+0)>U(\alpha_0-0)$.
 If not, $d=U(\alpha_0-0)-U(\alpha_0+0)>0$ and set
\[
	v(x) = \left\{
\begin{array}{lll}
	U(x) - d, & x \in (\alpha_0-\delta, \alpha_0) \\
	U(x), & x \notin (\alpha_0-\delta, \alpha_0).
\end{array}
\right.
\]
For a sufficiently small $\delta>0$, $v(x)>g(x)$; see Figure \ref{Fshi}.
\begin{figure}[tb]
\centering 
\includegraphics[keepaspectratio, scale=0.25]{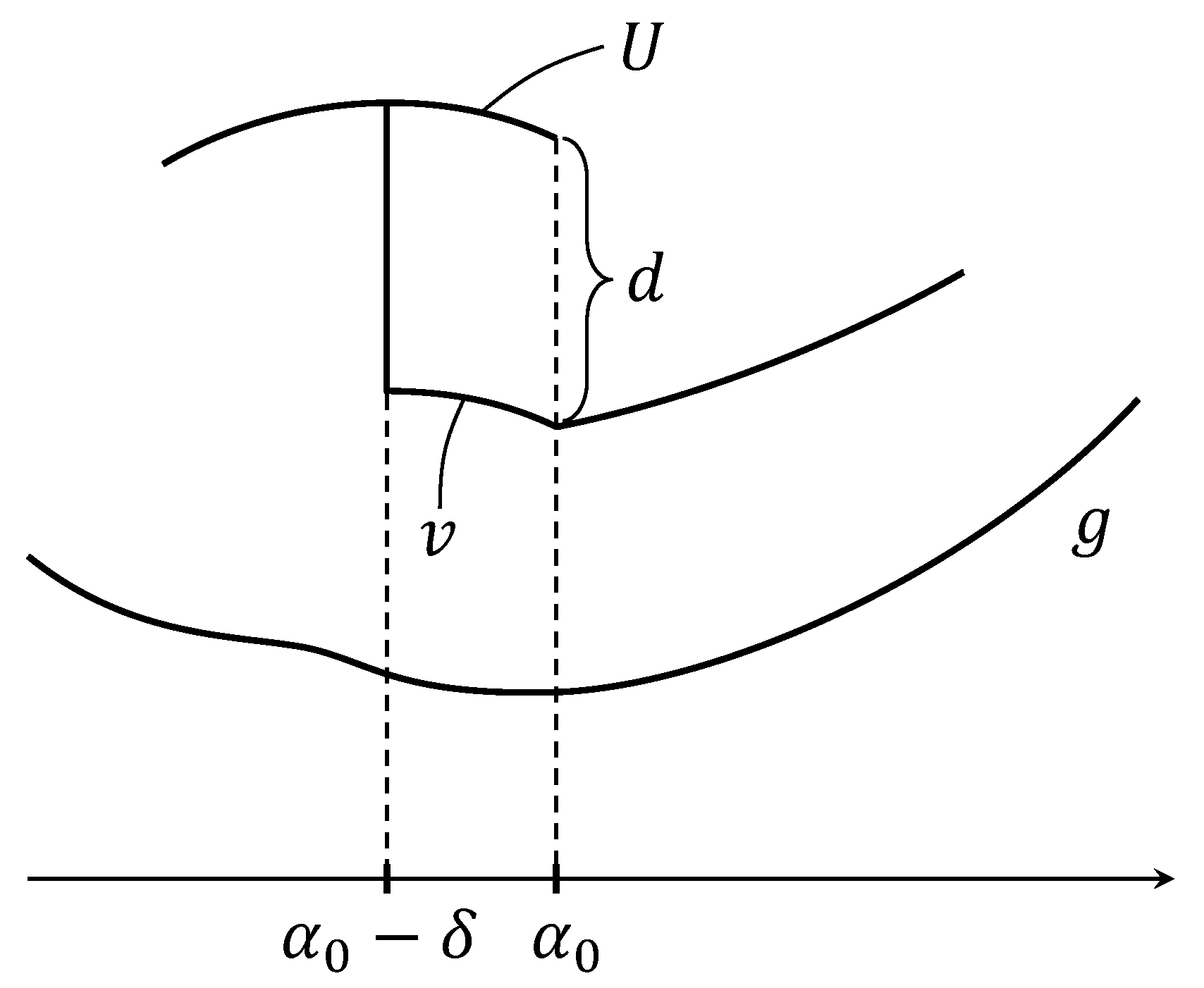}
\caption{shift of a jump\label{Fshi}}
\end{figure}
By definition, $TV_K(v)\leq TV_K(U)$ and $\mathcal{F}(v)<\mathcal{F}(U)$.
 Thus $U$ is not a minimizer of $TV_{Kg}$.
 Similarly, if there is a jump point $\beta_0$ of $U$ with $x_0<\beta_0$, then $U(\beta_0-0)>U(\beta_0+0)$.

We shall prove that $U$ is continuous on $(x_0,\beta)$ provided that $U>g$ on $(\alpha_0,\beta)$.
 If not, there would exist a jump point $\beta_0$ of $U$ with $\alpha_0<x_0<\beta_0<\beta$ satisfying $U(x)>g(x)$ for $x\in(\alpha_0,\beta_0)$, $U(\alpha_0+0)>g(\alpha_0)$, $U(\beta_0-0)>g(\beta_0)$.
 We shall prove that such configuration does not occur.
 As we see in the previous paragraph, $U$ must be constant on $(\alpha_0,\beta_0)$ so that $U(\alpha_0+0)=U(\beta_0-0)$.
 We replace the value if $U$ a constant $M$ smaller than $U(\alpha_0+0)=U(\beta_0-0)$ but close to its value so that the new function $U_M$ still has a property that $U_M>g$.
 Since $U(\alpha_0-0)<U(\alpha_0+0)$, $U(\beta_0+0)<U(\beta_0-0)$, clearly $TV_K(U_M)\le TV_K(U)$ because of (K1).
 Moreover since $U\ge U_M$ and $U>U_M$ on $(\alpha_0,\beta_0)$ and $U_M>g$, $\mathcal{F}(U_M)<\mathcal{F}(U)$.
 Thus, $U$ is not a minimizer.
 Thus we observe that $\beta_0=\beta$ and we conclude that $U$ is a constant on $[x_0,\beta]$.

Since there is a sequence of continuity point $x_j$ of $U$ converging to $\alpha_0$ as $j\to\infty$ keeping $x_j>\alpha_0$, we conclude that $U$ is a constant on $(x_j,\beta)$.
 This says that $U$ is constant on $(\alpha_0,\beta)$.
 To say that this corresponds to (i), it remains to prove that $U(\alpha-0)<g(\alpha)$.
 If not, $U(\alpha-0)\geq g(\alpha)$, then we take
\[
	v(x) = \left\{
\begin{array}{ll}
	\max \left( U(x) - d, \tilde{g}(x) \right), & x \in (\alpha_0, \alpha_0+\delta), \\
	U(x), & x \notin (\alpha_0, \alpha_0+\delta), \\
	U(x-0), & x = \alpha_0,
\end{array}
\right.
\]
where $d=U(\alpha_0+0)-U (\alpha_0-0)$.
 Here,
\[
	\tilde{g}(x) = \sup \left\{ g(y) \bigm|
	\alpha_0 \le y \le x \right\}
\]
which is continuous and non-decreasing.
 Moreover,
\[
	TV_K(v) \leq TV_K(U) + m\left(\tilde{g}(\delta)\right), \quad
	\mathcal{F}(v) \leq \mathcal{F}(U)
	- \delta \left(d-\tilde{g}(\delta)\right)^2,
\]
where $m(\sigma)=\sup\left\{ K(d+\tau)-K(d) \bigm| 0\leq\tau\leq\sigma\right\}$; see Figure \ref{Fmod}.
\begin{figure}[tb]
\centering 
\includegraphics[keepaspectratio, scale=0.25]{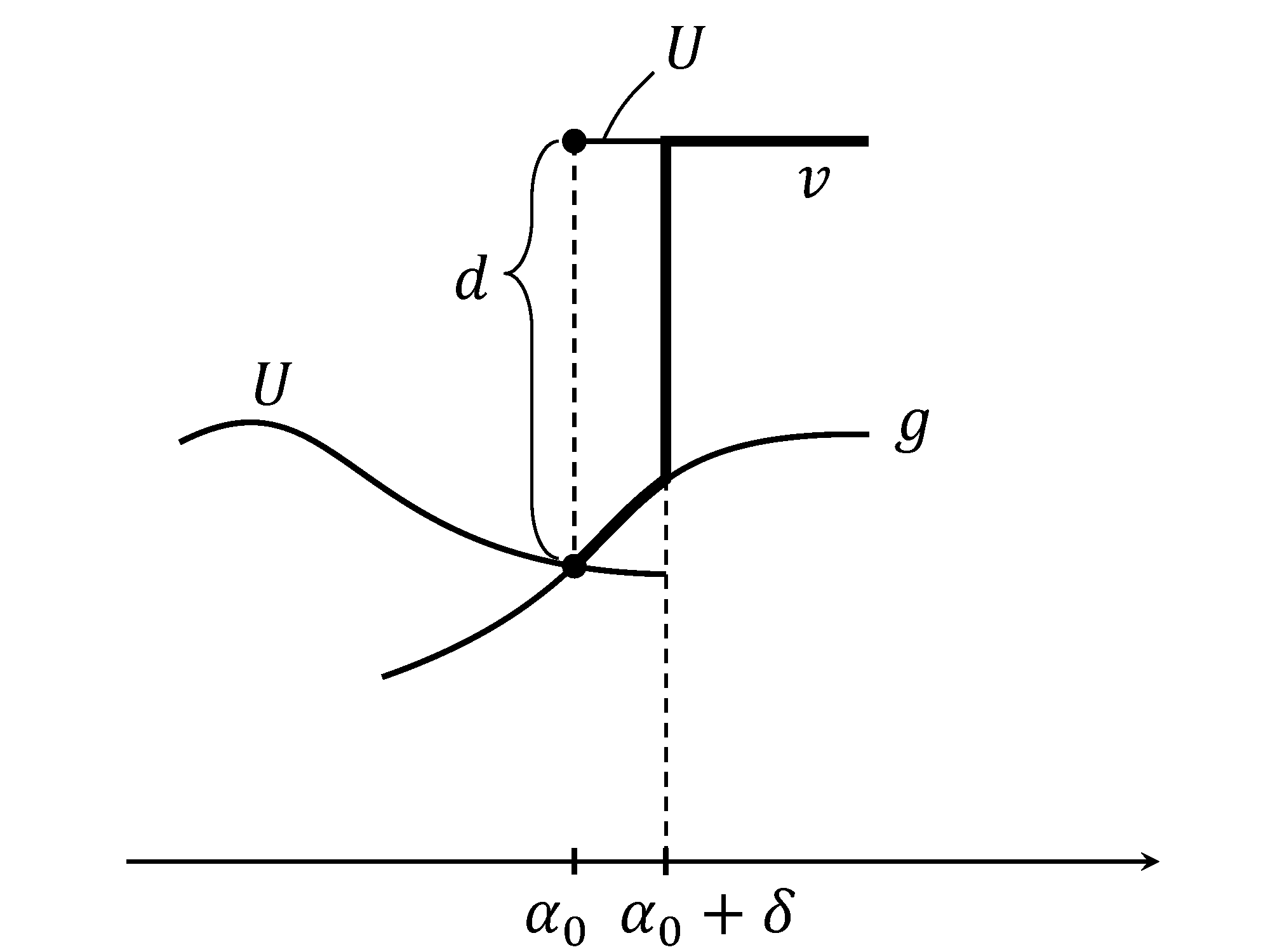}
\caption{modification of $U$\label{Fmod}}
\end{figure}
Thus
\[
	TV_{Kg}(v) \leq TV_{Kg}(U) + m\left(\tilde{g}(\delta)\right) - \delta\left(d-\tilde{g}(\delta)\right)^2.
\]
Since $\tilde{g}(\delta)\to0$ as $\delta\to0$, and $m(\sigma)\to0$ as $\sigma\to0$ by (K1), we conclude that for sufficiently small $\delta>0$, $TV_{Kg}(v)<TV_{Kg}(U)$.
 We now obtain (i).
 A symmetric argument yields (i\hspace{-1pt}i).
 The proof is now complete.
\end{proof}
For $u\in BV(a,b)$ and $g\in C[a,b]$, we set
\[
	C_\pm = \left\{ x \in [a, b] \bigm|
	u(x \pm 0) = g(x) \right\}.
\]
If $C_+\cap(a,b)=C_-\cap(a,b)$, we simply write 
\[
C=C_+\cup C_-
\]
and call $C$ the \emph{coincidence set} of $u$.
 By Lemma \ref{LPic}, we obtain a few properties of $C$.
\begin{lem} \label{LCoi}
Assume the same hypotheses of Lemma \ref{LPic} concerning $K$, $g$ and $U$.
 Let $C_\pm$ be defined for $u=U$.
 Then $C_-\cap(a,b)=C_+\cap(a,b)$, i.e., $U$ is continuous on the coincidence set $C$.
 If $(x_1,x_2)\cap C=\emptyset$ with $x_1,x_2\in C$, then $U$ is piecewise constant on $(x_1,x_2)$ with at most one jump.
\end{lem}
\begin{proof}
By Lemma \ref{LPic}, we easily see that $J_U \cap C_\pm=\emptyset$. 
 Thus, $U$ is continuous on $C$.

It remains to prove the second statement.
 By Lemma \ref{LPic}, the value of $U$ on $(x_1,x_2)$ is either $g(x_1)$ or $g(x_2)$.
 Moreover, if there are more than two jumps, $U$ must take value of $g$ at some point $x_*\in(x_1,x_2)$.
 In other words, $x_*\in C$.
 This contradicts to $C\cap(x_1,x_2)=\emptyset$.
\end{proof}
We conclude this section by showing that if two points $\alpha,\beta\in C$ with $\alpha<\beta$ for a minimizer $U$ is too close, then $U$ must be monotone in $(\alpha,\beta)$ under the assumption that $K$ satisfies (K1), (K2w) and (K3).

We first note comparison with usual total variation $TV$ and $TV_K$.
\begin{lem} \label{LCpTK}
Assume that $K$ satisfies (K1), (K2w) and (K3).
 If $u\in BV(\Omega)$ with $\Omega=(a,b)$ is continuous at $a$ and $b$, then
\[
	TV_K(u) \geq K(\rho) \quad\text{with}\quad
	\rho = \left\lvert u(b)-u(a) \right\rvert.
\]
\end{lem}
\begin{proof}
We may assume that $u(a)<u(b)$.
 By definition,
\[
	TV_K(u) = \int_{\Omega\backslash J_u} |Du|
	+ \sum_{x_i\in J_u} K(\rho_i)
\]
where $J_u$ denote the jump discontinuity of $u$ and $\rho_i=\left\lvert u(x_i+0)-u(x_i-0)\right\rvert$ for $x_i\in J_u$.
 We note that
\[
	\int_{\Omega\backslash J_u} |Du|
	\geq \left( \rho - \sum_{x_i\in J_u^+} \rho_i \right)_+,
\]
where $J_u^+=\left\{x\in J_u\bigm|u(x+0)-u(x-0)>0\right\}$.
 By subadditivity (K2w) and lower semicontinuity (K1), we see that
\[
	\sum_{x_i\in J_u^+} K(\rho_i)
	\geq K \left( \sum_{x_i\in J_u^+} \rho_i \right).
\]
Indeed,
\[
	K \left( \sum_{i=1}^\infty \rho_i \right)
	\le \varliminf_{m\to\infty} K \left( \sum_{i=1}^m \rho_i \right)
	\le \varliminf_{m\to\infty} \sum_{i=1}^m K (\rho_i) = \sum_{i=1}^\infty K(\rho_i).
\]

By (K2w), we have
\[
	K(\rho) \leq 2K(\rho/2) \leq \cdots
	\leq 2^m K(\rho/2^m) = \left(K(q)/q\right)\rho
\]
for $q=\rho/2^m$.
 Sending $m\to\infty$, we obtain by (K3) that
\[
	K(\rho) \leq \rho.
\]
We thus observe that
\begin{align*}
	TV_K(u) &\geq \int_{\Omega\backslash J_u} |Du|
	+ \sum_{x_i\in J_u^+} K(\rho_i)
	\geq \left( \rho - \sum_{x_i\in J_u^+} \rho_i \right)_+
	+ K \left( \sum_{x_i\in J_u^+} \rho_i \right) \\
	&\geq K \left(\left( \rho - \sum_{x_i\in J_u^+} \rho_i \right)_+\right) + K \left( \sum_{x_i\in J_u^+} \rho_i \right) \geq K(\rho).
\end{align*}
The last inequality follows from the subadditivity (K2w) and monotonicity of $K$ in (K1). 
\end{proof}

If we assume (K1) and (K3), we have, (K3w) by \eqref{EKEB}.
 We next give our monotonicity result.
\begin{thm} \label{TMon}
Assume that $K$ satisfies (K1), (K2w) and (K3), and that $g\in C[a,b]$.
 Let $U$ be a minimizer of $TV_{Kg}$ in $BV(a,b)$.
 Let $\alpha$ and $\beta$ with $\alpha<\beta$ be in $C$, where $C$ denotes the coincidence set.
 If $U(\alpha)\leq U(\beta)$ (resp.\ $U(\alpha)\geq U(\beta)$), then $U$ is non-decreasing (non-increasing) provided that $\beta-\alpha\le c_M/(\lambda M)$ if $\operatorname{osc}g=\max_{[\alpha,\beta]}g-\min_{[\alpha,\beta]}g\leq M$, where $c_M$ is a constant (K3w).
\end{thm}
\begin{proof}
We first note that $U$ is continuous at $\alpha,\beta$ by Lemma \ref{LCoi}.
 Moreover, if there is no point of $C$ in $(\alpha,\beta)$, by Lemma \ref{LCoi}, $U$ is piecewise constant with at most one jump in $(\alpha,\beta)$.
 Thus, $U$ is monotone.

We next consider the case that $C\cap(\alpha,\beta)\neq\emptyset$.
 We may assume that $U(\alpha)\leq U(\beta)$ since the proof for 
the other case is symmetric.
 We shall prove that
\[
	U(\alpha) \leq U(x_0) \leq U(\beta)
\]
for any $x_0\in C\cap(\alpha,\beta)$.
 Suppose that there were a point $x'_0 \in C\cap(\alpha,\beta)$ such that $U(\alpha)>U(x'_0)$ or $U(\beta)<U(x'_0)$.
 We may assume that $U(x'_0)>U(\beta)$ since the proof for the other case is parallel.
 By Lemma \ref{LPic} and Lemma \ref{LCoi}, the values of $U$ are attained in a coincidence set.
 Thus, there exists $x_0\in C\cap(\alpha,\beta)$ such that
\[
	U(x_0) = \max_{[\alpha,\beta]} U.
\]
We take $x_1\in C\cap(x_0,\beta]$ such that
\[
	U(x_1) = \min_{[x_0,\beta]} U;
\]
see Figure \ref{FChoi}.
\begin{figure}[tb]
\centering 
\includegraphics[keepaspectratio, scale=0.25]{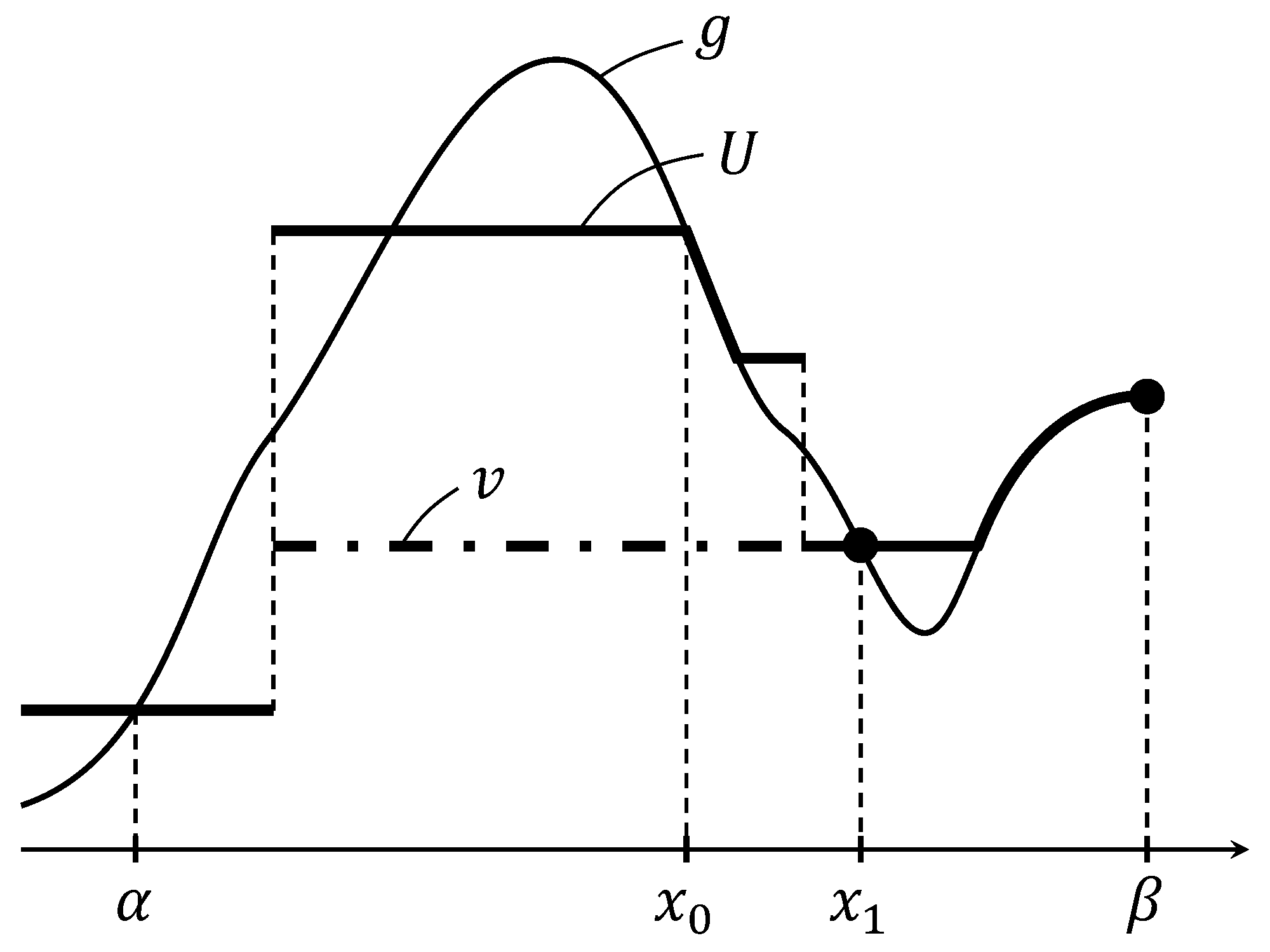}
\caption{\label{FChoi}}
\end{figure}
We now define a truncated function 
\[
	v(x) = \left\{
\begin{array}{ll}
	\min \left( U(x), U(x_1) \right), & x \in [\alpha, x_1] \\
	U(x), & x \in (x_1, \beta].
\end{array}
\right.
\]

We shall prove that $TV_{Kg}(v)<TV_{Kg}(U)$ if $\rho=U(x_0)-U(x_1)>0$.
 Since $U$ and $v$ are continuous at $x=x_0$ and $x_1$,
\begin{gather*}
	TV_K \left(U,(\alpha,\beta)\right) = TV_K \left(U,(\alpha,x_0) \right) 
	+ TV_K \left(U,(x_0,x_1) \right)
	+ TV_K \left(U,(x_1,\beta) \right), \\
	TV_K \left(v,(\alpha,\beta)\right)
	= TV_K \left(v,(\alpha,x_0) \right) 
	+ TV_K \left(v,(x_0,x_1) \right)
	+ TV_K \left(v,(x_1,\beta) \right).
\end{gather*}
Since $v=U$ near $\alpha$ and $\beta$ and also outside $(\alpha,\beta)$, the values $TV_K$ of $U$ and $u$ outside $(\alpha,\beta)$ are the same.
 Since $TV_K$ does not increase by truncation, we see
\[
	TV_K \left(U,(\alpha,x_0) \right) 
	\geq TV_K \left(v,(\alpha,x_0) \right).
\]
Since $v=U$ on $(x_1,\beta)$, we proceed
\begin{equation*}
	TV_K(U) - TV_K(v) 
	\geq TV_K \left(U,(x_0,x_1) \right)
	- TV_K \left(v,(x_0, x_1) \right)
	= TV_K \left(U,(x_0,x_1) \right).
\end{equation*}
By Lemma \ref{LCpTK}, we observe that
\[
	TV_K \left(U,(x_0,x_1) \right) 
	\geq K(\rho) \quad\text{with}\quad
	\rho = U(x_0) - U(x_1).
\]
Thus by (K3w), we now obtain
\[
	TV_K (U) - TV_K(v) > c_M\rho.
\]
We next compare the values $\mathcal{F}(v)$ and $\mathcal{F}(U)$.
 By definition, we observe
\begin{align*}
	\frac2\lambda \left(\mathcal{F}(v) - \mathcal{F}(U)\right)
	&= \int_\alpha^{x_1} (g-v)^2\; dx - \int_\alpha^{x_1} (g-U)^2\; dx \\
	& \leq \int_\alpha^{x_1} \rho\left( |g-v|+|g-U|\right) dx
    \leq (\beta-\alpha) \rho 2M.
\end{align*}
The last inequality follows from the fact that the value of $U$ must be between $\min_{[\alpha,\beta]}g$ and $\max_{[\alpha,\beta]}g$.
 We thus observe that
\begin{equation*}
    TV_{Kg} (U) - TV_{Kg}(v)
    > c_M\rho - \lambda(\beta-\alpha)\rho M
	=\left( c_M - \lambda(\beta-\alpha)M \right)\rho.
\end{equation*}
If $\beta-\alpha$ satisfies $\beta-\alpha\leq c_M/(\lambda M)$, $U$ is not a minimizer provided that $\rho>0$, i.e., $U(x_0)>U(x_1)$.
 Thus, we conclude that $U(x_0)\leq U(\beta)$.

So far we have proved that $U$ is non-decreasing on $C\cap[\alpha,\beta]$.
 By Lemma \ref{LCoi}, we conclude that $U$ itself is non-decreasing in $[\alpha,\beta]$.
\end{proof}
\begin{remark} \label{RK3w}
	In Theorem \ref{TMon}, if we assume that $U$ is a piecewise constant function (possibly with infinitely many jumps), we may assume (K3w) instead of (K3) since the conclusion of Lemma \ref{LCpTK} holds for piecewise constant functions without assuming (K3).
\end{remark}

\section{Minimizers for general one-dimensional data} \label{SGe} 

In this section, we shall prove that a minimizer $U$ is piecewise constant if $K$ satisfies (K1), (K3) and (K2) instead of (K2w).
 In other words, we shall prove our main result.

If we assume (K2), then merging jumps decrease the value $TV_K$.
 However, $\mathcal{F}$ may increase.
 We have to estimate an increase of $\mathcal{F}$.

\subsection{Bound for an increase of fidelity} \label{SEFid} 

We shall estimate an increase of fidelity $\mathcal{F}$.
 We begin with a simple setting.
 We set  for $\gamma\in(\alpha,\beta)$,
\[
	U_0^\gamma(x) = \left\{
\begin{array}{ll}
	g(\alpha), & x \in [\alpha,\gamma), \\
	g(\beta), & x \in [\gamma,\beta];
\end{array}
\right.
\] 
see Figure \ref{FUGamma}. 
\begin{figure}[tb]
\centering
\includegraphics[keepaspectratio, scale=0.25]{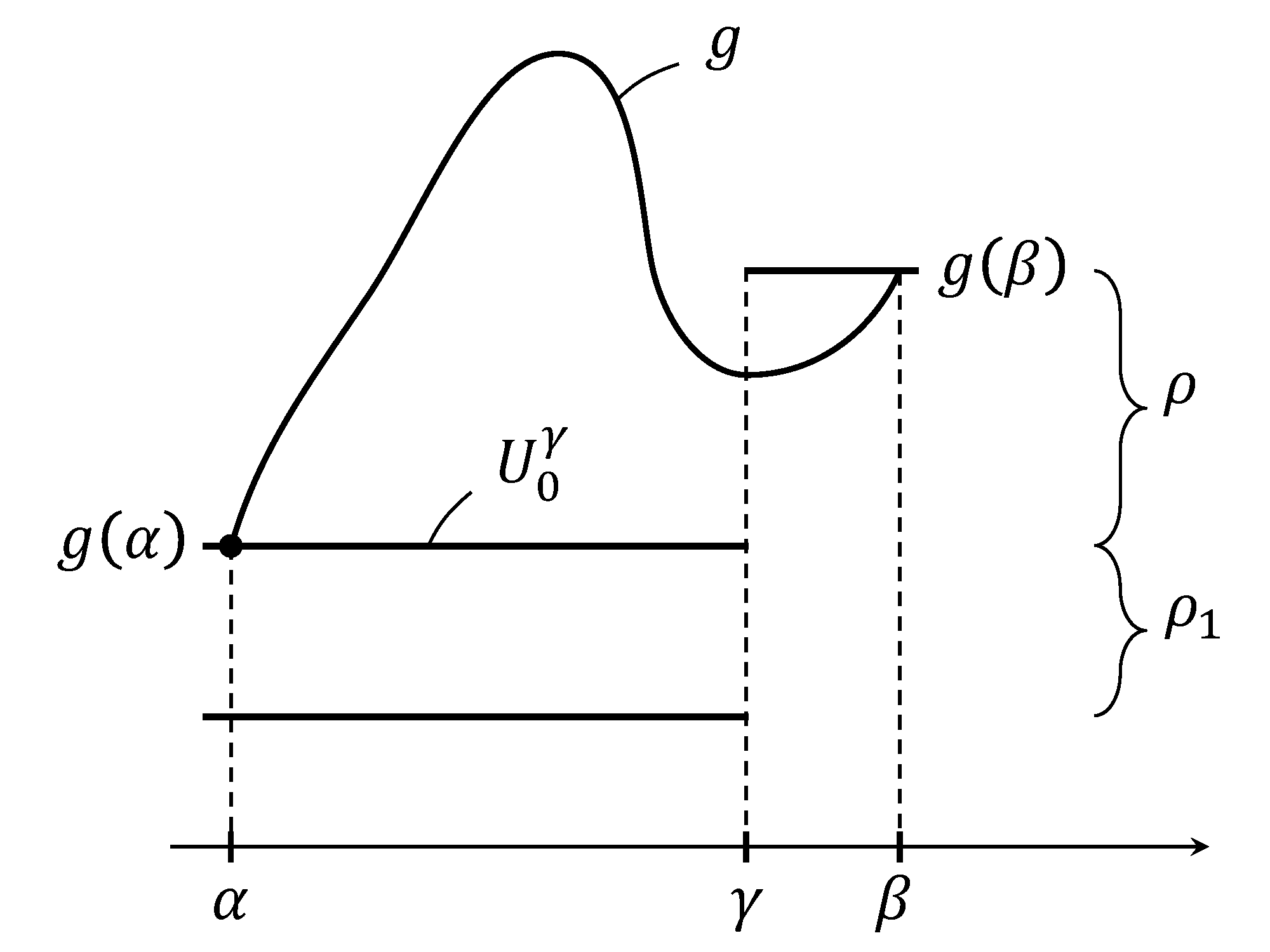}
\caption{profile of $U_0^\gamma$ and $g$\label{FUGamma}}
\end{figure}
The fidelity of $U_0^\gamma$ on $(\alpha,\beta)$ is denoted by $\lambda F(\gamma)/2$, i.e.,
\[
	F(\gamma) := 
	\int_\alpha^\beta \left| U_0^\gamma - g \right|^2 dx
\]
for $\gamma\in[\alpha,\beta]$.
 Since we do not assume that $g$ is non-decreasing, $g$ may be very large on $(\alpha,\gamma)$.
 Fortunately, we observe that $g$ cannot be too large on $(\alpha,\gamma)$ in the average if $F(\gamma)$ is smaller than $F(\alpha+0)$.
\begin{prop} \label{PCong}
Assume that $g\in C[\alpha,\beta]$ and that $U_0^\gamma(\alpha)=g(\alpha)$, $U_0^\gamma(\beta)=g(\beta)$ with $\rho= U_0^\gamma (\beta)- U_0^\gamma (\alpha)>0$.
 If $\gamma\in[\alpha,\beta]$ satisfies $F(\alpha+0)\geq F(\gamma)$, then
\[
	\int_\alpha^\gamma g(x)\; dx \leq \frac12 \left( U_0^\gamma(\alpha) + U_0^\gamma(\beta) \right) (\gamma - \alpha)
\]
or
\[
	\int_\alpha^\gamma \left( g(x) - U_0^\gamma(\alpha) \right) dx \leq \rho(\gamma - \alpha)/2.
\]
\end{prop}
\begin{proof}
We observe that
\begin{align*}
	F(\alpha + 0) - F(\gamma)
	&= \int_\alpha^\gamma \left\{ \left( g(\beta) - g(x)\right)^2 - \left( g(\alpha) - g(x)\right)^2 \right\} dx \\
	&= -\int_\alpha^\gamma \rho \left\{ 2g - \left(g(\alpha) + g(\beta)\right) \right\} dx.
\end{align*}
Since $\rho>0$, $F(\alpha+0)-F(\gamma)\geq0$ implies that
\[
	\int_\alpha^\gamma 2g\; dx
	\leq \int_\alpha^\gamma \left( g(\alpha) + g(\beta)\right) dx
	= \left( U_0^\gamma(\alpha) + U_0^\gamma(\beta)\right)(\gamma-\alpha).
\]
\end{proof}

We give a simple application.
 See Figure \ref{FUGamma}.
\begin{lem} \label{LPert}
Assume the same hypotheses of Proposition \ref{PCong}.
 Then,  for $\rho_1>0$, 
\[
	\int_\alpha^\gamma \left( U_0^\gamma - \rho_1 - g \right)^2 dx
	-\int_\alpha^\gamma \left( U_0^\gamma - g \right)^2 dx
	\leq \rho_1(\rho_1 + \rho) (\gamma-\alpha).
\]
\end{lem}
\begin{proof}
We may assume that $U_0^\gamma(\alpha)=0$ by adding a constant to both $U_0^\gamma$ amd $g$.
 The left-hand side equals
\[
	I = \int_\alpha^\gamma ( \rho_1 + g )^2\; dx
	-\int_\alpha^\gamma g^2\; dx
	= \rho_1 \int_\alpha^\gamma \{ \rho_1 + 2g\} \; dx.
\]
Since
\[
	2 \int_\alpha^\gamma g(x)\; dx
	\leq \rho(\gamma-\alpha)
\]
by Proposition \ref{PCong}, we end up with $I\leq(\rho_1^2+\rho_1\rho)(\gamma-\alpha)$.
\end{proof}
\begin{remark} \label{RCong}
In Proposition \ref{PCong} and Lemma \ref{LPert}, we do not assume that $g< U_0^\gamma (\beta)$ on $(\gamma,\beta)$ nor $g\geq U_0^\gamma (\alpha)$ on $(\alpha,\gamma)$.
\end{remark}
We next consider behavior of $g$ between two points of the coincidence set where $U$ is a constant.
\begin{prop} \label{PFEC}
Assume that $K$ satisfies (K1), (K2w) and (K3).
 Assume that $g\in C[a,b]$.
 Let $U\in BV(a,b)$ be a minimizer of $TV_{Kg}$.
 Let $\alpha,\beta\in[a,b)$ be $\alpha<\beta$ and $\alpha,\beta\in C$.
 Assume that $U$ is non-decreasing and there is $\gamma\in[\alpha,\beta)\cap C$ such that $U(\gamma)=U(\alpha)=g(\alpha)$ and $U(x)>U(\gamma)$ for $x>\gamma$.
 Assume that there is $p_j\in U(C)$ such that $U(\gamma)<p_j<U(\beta)$, $p_j\downarrow U(\gamma)$ as $j\to\infty$.
 Then, $\int_\alpha^\gamma\left(g(x)-U(\alpha)\right)dx\leq0$.
\end{prop}
\begin{proof}
We may assume that $U(\alpha)=g(\alpha)=0$ so that $\lim_{j\to\infty}p_j=0$.
 We set 
\[
	v_j(x) = \left\{
\begin{array}{ll}
	\max \left(p_j,U(x)\right), & \alpha_j := \alpha + (\gamma-\alpha)/j< x \\
	U(x), & x \leq \alpha_j,
\end{array}
\right.
\]
 for $j\geq2$; see Figure \ref{FLFEC}. 
\begin{figure}[tb]
\centering
\includegraphics[keepaspectratio, scale=0.25]{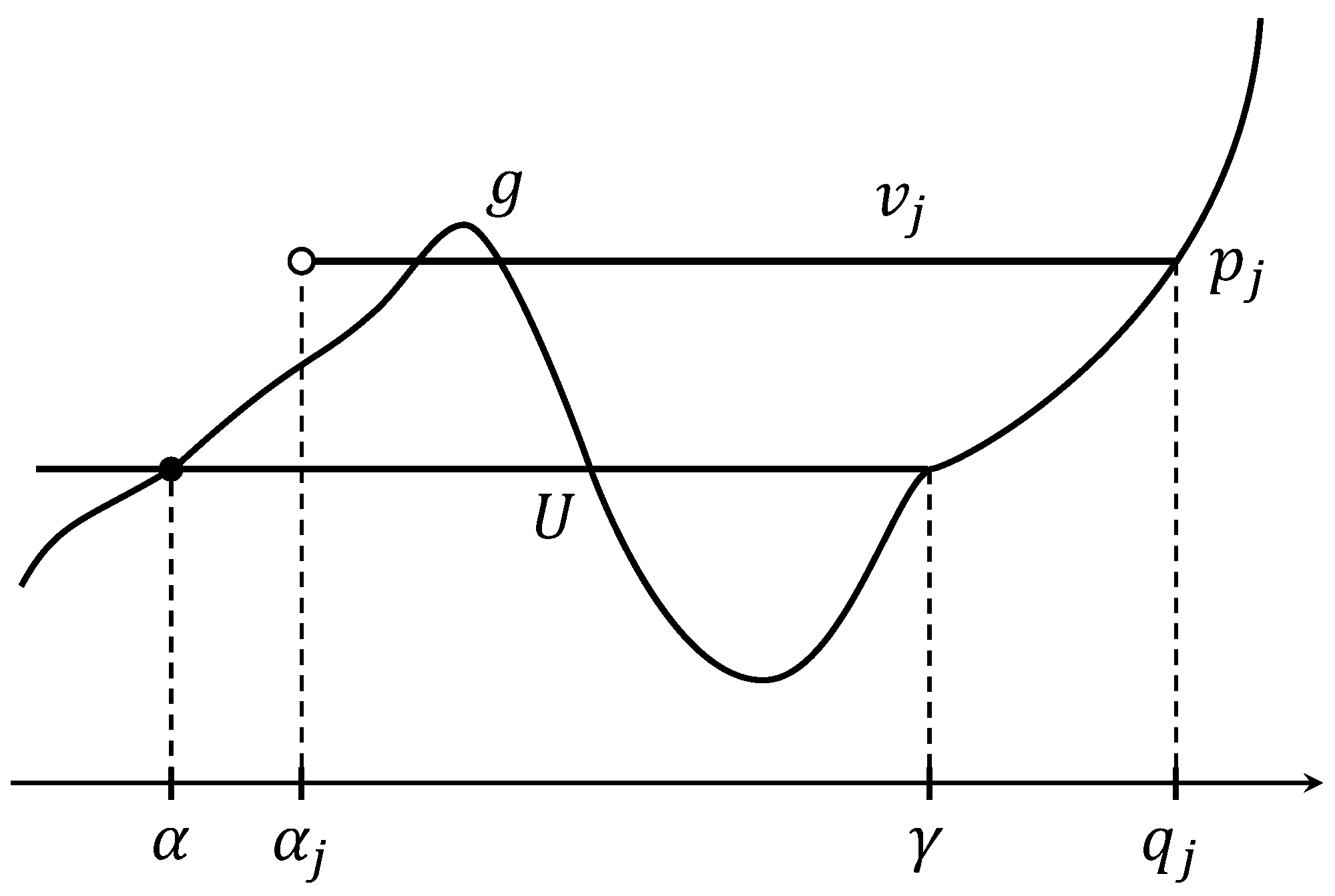}
\caption{$v_j$ and $U$\label{FLFEC}}
\end{figure}
Since $U$ is a minimizer, by definition,
\[
	TV_{Kg}(v_j) \geq TV_{Kg}(U).
\]
By Lemma \ref{LCpTK},
\begin{align*}
	TV_K \left(U,(\alpha,\beta)\right)
	& \geq TV_K \left(U,(\alpha,\gamma)\right)+ TV_K \left(U,(\gamma,q_j)\right) + TV_K \left(U,(q_j,\beta)\right) \\
	&\geq 0 + K(p_j) + TV_K \left(U,(q_j,\beta)\right),
\end{align*}
where $q_j\in U^{-1}(p_j)$.
Since
\[
	TV_K \left(v_j,(\alpha,\beta)\right) = K(p_j) + TV_K \left(U,(q_j,b)\right),
\]
$TV_{Kg}(v_j)\geq TV_{Kg}(U)$ implies that $\mathcal{F}(v_j)\geq\mathcal{F}(U)$.
 In other words,
\[
	\int_\alpha^{q_j} \left\{(v_j-g)^2 - (U-g)^2 \right\} dx \geq 0.
\]
Dividing the region of integration $(\alpha,q_j)$ into $(\alpha,\gamma)$ and $(\gamma,q_j)$, we obtain
\[
	\int_{\alpha_j}^\gamma \left\{ g^2-(p_j-g)^2\right\} dx
	\leq \int_\gamma^{q_j} \left\{ (p_j-g)^2 - (U-g)^2\right\} dx,
\]
or
\[
	p_j \int_{\alpha_j}^\gamma (2g-p_j)\; dx
	\leq \int_\gamma^{q_j} (p_j-U) (p_j+U-2g)\; dx.
\]
Since $U\leq p_j$ on $(\gamma,q_j)$, the right-hand side is dominated by
\[
	p_j \int_\gamma^{q_j} |p_j + U - 2g|\; dx
	\leq p_j |q_j - \gamma| \left( 2p_j + 2\|g\|_\infty \right).
\]
Thus
\[
	\int_{\alpha_j}^\gamma (2g-p_j)\; dx 
	\leq |q_j-\gamma| \left( 2p_j + 2\|g\|_\infty \right).
\]
Sending $j\to\infty$ yields that
\[
	\int_\alpha^\gamma g\; dx \leq 0,
\]
since $q_j\to\gamma$ by our assumption that $U(x)>0$ for $x>\gamma$ and $U$ is non-decreasing.
 The proof is now complete.
\end{proof}

We say that a closed interval $F$ is a \emph{facet} of $U$ if $F$ is a maximal nontrivial closed interval such that $U$ is a constant on the interior $\operatorname{int}F$ of $F$.
 Let $|F|$ denote its length.
 We are able to claim a similar statement for each facet of a minimizer $U$.
\begin{lem} \label{LCGF}
Assume that $K$ satisfies (K1), (K2w) and (K3).
 Assume that $g\in C[a,b]$.
 Let $U\in BV(a,b)$ be a minimizer of $TV_{Kg}$.
 Assume that $U$ is non-decreasing.
 Let $F=[x_0,x_1]$ with $x_1<b$ be a facet of $U$.
 Then
\begin{equation} \label{Ekey}
	\int_F \left( g(x)-U \right) dx
	\leq \left( U(x_1+0) - U(x_1-0) \right)|F|/2.
\end{equation}
\end{lem}
\begin{proof}
We may assume $U\equiv0$ on $F$.
 By Lemma \ref{LPic}, $F\cap C\neq\emptyset$.
 We set
\[
	\alpha = \inf(F\cap C)
\]
which is still in $C$ since $C$ is closed.
 By Lemma \ref{LPic},
\[
	g(x) < U(x) = 0 \quad\text{on}\quad [x_0, \alpha)
\]
since $U$ is non-decreasing.
 If $U(x_1+0)=U(x_1-0)$, then $x_1\in C$ by Lemma \ref{LCoi}.
 Moreover, there is $p_j\in U(C)$ such that $p_j\downarrow U(\gamma)$ with $p_j>0$ since otherwise it would contradict the maximality of $F$ by Lemma \ref{LCoi}.
 Thus by Proposition \ref{PFEC},
\[
	\int_\alpha^{x_1} g(x)\; dx \leq 0.
\]
Thus, we obtain \eqref{Ekey} when $U$ does not jump at $x_1$ since we know $g<0$ on $[x_0,\alpha)$.

If $U(x_1+0)-U(x_1-0)>0$, we may apply Proposition \ref{PCong} and conclude that
\[
	\int_\alpha^{x_1} g\; dx
	\leq U(x_1+0)(x_1-\alpha)/2.
\]
Since we know that $g<0$ on $[x_0,\alpha)$, the proof of Lemma \ref{LCGF} is now complete.
\end{proof}
\begin{thm} \label{TFidInc}
Assume that $K$ satisfies (K1), (K2w) and (K3).
 Assume that $g\in C[a,b]$.
 Let $U\in BV(a,b)$ be a minimizer of $TV_{Kg}$.
 Assume that $U$ is non-decreasing.
 Let $\alpha,\beta\in C$ with $\alpha<\beta<b$.
 Then
\[
	\int_\alpha^\beta \left( U(\alpha) - g \right)^2 dx
	- \int_\alpha^\beta (U-g)^2\; dx
	\leq \rho^2 (\beta-\alpha), 
\]
where $\rho=U(\beta)-U(\alpha)\geq0$.
\end{thm}
\begin{proof}
If $\rho=0$, a minimizer must be constant $U(\alpha)$ so the above inequality is trivially fulfilled.
 We may assume $\rho>0$.

As before, we may assume $U(\alpha)=0$ so that $U(x)\geq0$.
 We proceed
\[
	\int_\alpha^\beta g^2\; dx
	- \int_\alpha^\beta (U-g)^2\; dx
	=\int_\alpha^\beta U(2g-U)\; dx.
\]
On the coincidence set $C$,
\[
	\int_C U(2g-U)\; dx = \int_C U^2\, dx
\]
since $g=U$ on $C$.
 On a facet $F=[x_0,x_1]$, by Lemma \ref{LCGF},
\begin{align*}
	\int_F U(2g-U)\;dx &= U(x_1-0) \int_F \left\{ \left( 2g-2U(x_1-0)\right) + U(x_1-0) \right\} dx \\
	&\leq U(x_1-0) \left( \left( U(x_1+0) - U(x_1-0) \right) + U(x_1-0) \right) |F| \\
	&= U(x_1-0) U (x_1+0) |F| \\
	&\leq \rho^2|F|
\end{align*}
since $x_1\leq\beta<b$.
 By Lemma \ref{LPic} we know that $(\alpha,\beta)=\bigcup_{i=1}^\infty F_i\cup C$ with at most countably many facets $\{F_i\}$.
 Thus, above estimates on $C$ and $F$ yield
\begin{align*}
	\int_\alpha^\beta U(2g-U)\;dx
	&\leq \sum_{i=1}^\infty \int_{F_i} U(2g-U)\;dx
	+ \int_C U(2g-U)\;dx \\
	&\leq \rho^2 \sum_{i=1}^\infty |F_i| + \int_C U^2\;dx
     \leq \rho^2 \sum_{i=1}^\infty |F_i| + \rho^2|C|.
\end{align*}
We now conclude that
\[
	\int_\alpha^\beta g^2\;dx - \int_\alpha^\beta (U-g)^2\;dx
	\leq \rho^2 \left(\sum_{i=1}^\infty |F_i| + |C|\right)
	= \rho^2(\beta-\alpha). 
\]
\end{proof}

\subsection{No possibility of fine structure} \label{SSNP}

Our goal in this subsection is to prove our main theorem.
 In other words, we shall prove that a minimizer does not allow to have a ``fine'' structure under the assumption (K2).
 At the end of this subsection, we prove our main theorem Theorem \ref{TMain}.
\begin{lem} \label{LLEN}
Assume that $K$ satisfies (K1), (K2) and (K3).
 Assume that $g\in C[a,b]$.
 Let $U\in BV(a,b)$ be a minimizer of $TV_{Kg}$.
 Assume that $U$ is non-decreasing.
 Let $\alpha,\beta,\gamma\in C$ satisfy $a\leq\alpha<\gamma<\beta\leq b$ and $U(\beta)-U(\gamma)\geq U(\gamma)-U(\alpha)>0$.
 Assume that $(\gamma,\beta)\cap C=\emptyset$.
 Then
\[
	\beta-\alpha > C_M/\lambda
\]
for $M\geq\operatorname{osc}_{[\alpha,\beta]}g$, where $C_M$ is the constant in (K2).
\end{lem}
\begin{proof}
We set $\rho_1= U(\gamma) -U(\alpha)$, $\rho_2=U(\beta)-U(\gamma)$.
 Since $(\gamma,\beta) \cap C=\emptyset$, by Lemma \ref{LCoi}, $U$ has only one jump at $x_1\in(\gamma,\beta)$ and $U=U(\beta)$ for $x\in(x_1,\beta)$, $U=U(\gamma)$ for $x\in(\gamma,x_1)$.

We set
\[
	v(x) = \left\{
\begin{array}{ll}
	U(\alpha), & x \in (\alpha, x_1), \\
	U(x), & x \notin (\alpha, x_1)
\end{array}
\right.
\]
and compare $TV_K(v)$ with $TV_K(U)$ on $(\alpha,\beta)$; see Figure \ref{FGUV}.
\begin{figure}[tb]
\centering 
\includegraphics[keepaspectratio, scale=0.25]{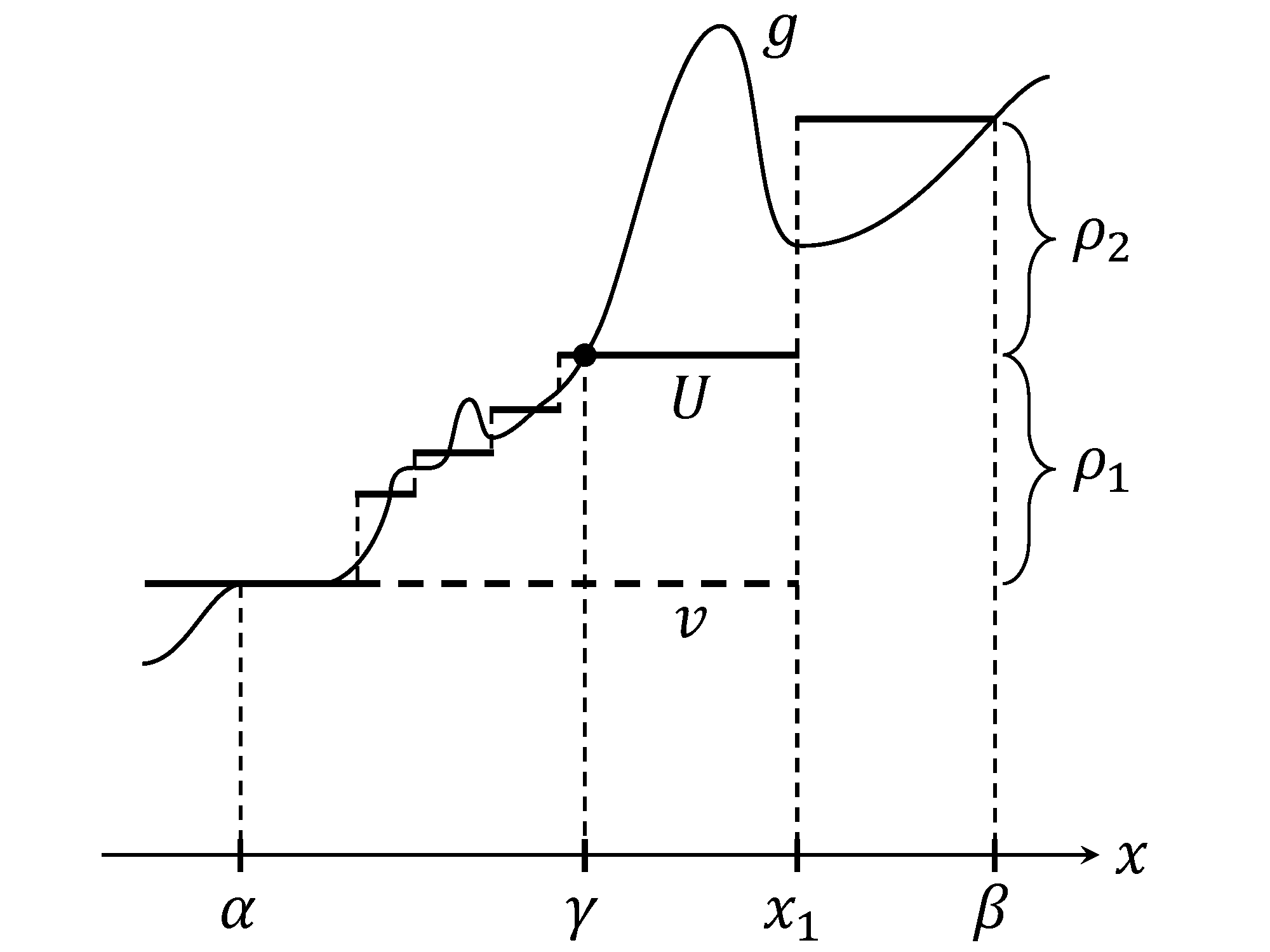}
\caption{the graph of $U$, $g$ and $v$\label{FGUV}}
\end{figure}
Then
\begin{equation*}
	TV_K \left(U,(\alpha,\beta)\right)
    = K(\rho_2) + TV_K\left(U,(\alpha,\gamma)\right)
	\geq K(\rho_2) + K(\rho_1)
\end{equation*}
by Lemma \ref{LCpTK}.
 Clearly,
\[
	TV_K \left(v,(\alpha,\beta)\right) = K(\rho_1+\rho_2).
\]
By (K2), we see that
\begin{equation*}
	TV_K \left(U,(\alpha,\beta)\right) - TV_K\left(v,(\alpha,\beta)\right)  
	\geq K(\rho_1) + K(\rho_2) - K(\rho_1+\rho_2)
	\geq C_M \rho_1 \rho_2.
\end{equation*}
Since $U$ minimizes $\int_\gamma^\beta|U-g|^2\;dx$ for fixed $\rho_1$ and $\rho_2$, we have, by Lemma \ref{LPert},
\[
	\int_\gamma^{x_1} \left\{ (U-g)^2 - (v-g)^2 \right\} dx \geq -\rho_1(\rho_1+\rho_2)(x_1-\gamma).
\]
Since $\gamma<x_1<b$ and $U$ is a minimizer of $TV_{Kg}$ on $(\alpha,x_1)$, we have, by Theorem \ref{TFidInc}, 
\[
	\int_\alpha^\gamma \left\{ (U-g)^2 - (v-g)^2 \right\} dx \geq - \rho_1^2(\gamma-\alpha).
\]
We now conclude that
\[
	TV_{Kg}(U) - TV_{Kg}(v) \geq C_M \rho_1\rho_2
	- \left(\rho_1(\rho_1+\rho_2)(x_1-\gamma) + \rho_1^2(\gamma-\alpha) \right) \lambda\bigm/ 2.
\]
Since we assume that $\rho_1\leq\rho_2$, this implies that
\[
	TV_{Kg}(U) - TV_{Kg}(v) \geq \rho_1\rho_2 \left(C_M-(x_1-\alpha) \lambda \right).
\]
Since $U$ is a minimizer, $C_M-(x_1-\alpha)\lambda\leq0$, we conclude that
\[
	\beta-\alpha > x_1 - \alpha \geq C_M/\lambda.
\]
\end{proof}
From the proof of Lemma \ref{LLEN}, we have a rather general estimate for $U_0^\gamma$ defined at the beginning of Section \ref{SEFid}.
\begin{lem} \label{LGEs}
Assume that $K$ satisfies (K1), (K2) and (K3).
 Assume that $g\in C[\alpha,\beta]$ and $M\geq\operatorname{osc}_{[\alpha,\beta]}g$.
 Assume that $\rho=g(\beta)-g(\alpha)>0$ and $x_1\in(\alpha,\beta)$.
 Let $U$ be a non-decreasing function with $U(\alpha)=g(\alpha)$, $U(\beta)=g(\beta)$ which is continuous at $\alpha$ and $\beta$.
 Assume that $(\alpha,\beta)=\bigcup_{i=1}^\infty F_i\cup C$ where $F_i$ is a facet and $C$ is a coincidence set with $F_1=[x_1,\beta]$ and $F_2=[x_0,x_1]$, $x_0\in(\alpha,x_1)$.
 (The set $F_i$ for $i\geq 3$ could be empty.)
 Assume that $U(\beta-0)=g(\beta)$ and $g(\beta)-U(x_1-0):=\rho_2>\rho/2$.
 Assume further that $U$ satisfies \eqref{Ekey} on each $F_i$ for $i\geq3$.
 Assume that $F(x_1)\leq F(\alpha+0)$ for $U_0^{x_1}$. 
 If $C$ contains an interior point of $F_2$, then
\[
	TV_{Kg}(U) - TV_{Kg}(U_0^{x_1})
	\geq  \rho_1 \rho_2 \left( C_M - (x_1 - \alpha)\lambda \right),
\]
with $\rho_1=\rho-\rho_2$.
\end{lem}
We are interested in the case that there are no jumps.
 We begin with an elementary property of $K$.
\begin{lem} \label{LEeK}
Assume that $K$ satisfies (K2) and (K3).
 Then, $\rho-K(\rho)\geq C_M\rho^2/2$.
\end{lem}
\begin{proof}
An iterative use of (K2) yields
\begin{align*}
	K(\rho) &\leq 2K \left(\frac\rho2\right) - C_M\left(\frac\rho2\right)^2
	\leq 2 \left( 2K \left(\frac\rho4\right) - C_M\left(\frac\rho4\right)^2 \right) - C_M\left(\frac\rho2\right)^2 \\
	&\leq 2^m K \left(\frac{\rho}{2^m}\right) - C_M \rho^2 \sum_{j=1}^m 2^{j-1}\left(\frac{1}{2^j}\right)^2
\end{align*}
for $m=1,2,\ldots$.
 Since $\sum_{j=1}^m 2^{-j-1}\to1/2$ as $m\to\infty$, sending $m\to\infty$ yields
\[
	K(\rho) \leq \rho - C_M\rho^2/2
\]
by (K3).
 The proof is now complete.
\end{proof}
\begin{lem} \label{LCont}
Assume the same hypotheses of Lemma \ref{LLEN} concerning $K$ and $g$.
 Let $U\in BV(a,b)$ be a minimizer of $TV_{Kg}$.
 Assume that $U$ is non-decreasing and continuous on $[\alpha,\beta]\subset[a,b]$ with $\alpha,\beta\in C$.
 Then $U(\alpha)=U(\beta)$.
\end{lem}
\begin{proof}
Suppose that $U(\alpha)\neq U(\beta)$ so that $U(\alpha)<U(\beta)$, there would exist at least one $p\in\left(U(\alpha),U(\beta)\right)$ such that $U^{-1}(p)$ is a singleton $\{x_0\}$ since $U$ is continuous.
 By Lemma \ref{LPic}, $x_0\in C$.
 Moreover, there is a sequence $p_j\downarrow p$ ($j\to\infty$) such that $U^{-1}(p_j)$ is a singleton $\{x_j\}$.
 This is because the set of values $q$ where $U^{-1}(q)$ is not a singleton is at most a countable set.
 Since $U^{-1}(p)$ is a singleton, $x_j\downarrow x_0$.
 Again by Lemma \ref{LPic}, $x_j\in C$.
 We set
\[
	v_j(x) = \left\{
\begin{array}{ll}
	p, & x \in (x_0, x_j), \\
	U(x), & x \notin (x_0, x_j).
\end{array}
\right.
\]
 By Theorem \ref{TFidInc}, we obtain
\[
	\frac2\lambda \left( \mathcal{F}(v_j) - \mathcal{F}(U) \right) \leq \rho_j^2(x_j-x_0)
\]
with $\rho_j=p_j-p$.

Since $TV=TV_K$ for a continuous function, we see that
\[
	TV_K \left( U,(x_0,x_j) \right) = \rho_j.
\]
By Lemma \ref{LEeK}, we observe that
\[
	TV_K(U) - TV_K(v_j) = \rho_j - K(\rho_j) \geq C_M\rho_j^2/2, \quad
	M > \operatorname{osc}_{[\alpha,\beta]}g.
\]
We thus conclude that
\begin{equation*}
	TV_{Kg}(U) - TV_{Kg}(v_j) 
	\geq \left(C_M \rho_j^2 - \lambda \rho_j^2(x_j-x_0)\right)/2
	= \rho_j^2 \left( C_M - \lambda (x_j-x_0) \right)/2.
\end{equation*}
For a sufficiently large $j$, $C_M-\lambda(x_j-x_0)>0$ since $x_j\downarrow x_0$.
 This would contradict to our assumption that $U$ is a minimizer.
 Thus $U(\alpha)=U(\beta)$.
\end{proof}
\begin{thm} \label{TDis}
Assume that $K$ satisfies (K1), (K2) and (K3).
 Assume that $g\in C[a,b]$.
 Let $U\in BV(a,b)$ be a minimizer of $TV_{Kg}$.
 Let $\alpha,\beta\in C$ with $a\leq\alpha<\beta\leq b$.
 Assume that $\beta-\alpha\leq A_M/\lambda$ with $A_M=\min\{c_M/M, C_M \}$ and $\operatorname{osc}_{[\alpha,\beta]}g\leq M$, where $c_M$ is in (K3w) and $C_M$ is in (K2).
 Then the values of $U$ on $[\alpha,\beta]$ are either $U(\alpha)$ and $U(\beta)$ and $U$ has at most one jump point in $(\alpha,\beta)$.
\end{thm}

\begin{proof}
By Theorem \ref{TMon}, $U$ is non-decreasing in $[\alpha,\beta]$.
 We may assume $U(\alpha)<U(\beta)$.
 If there is no jump, i.e., $U\in C[\alpha,\beta]$, by Lemma \ref{LCont}, $U$ is not a minimizer so $U$ must have at least one jump point in $(\alpha,\beta)$.
 We take a jump point $x_0$ such that jump size
\[
	U(x_0+0) - U(x_0-0)\ (>0)
\]
is maximum among all jump size of $U$ in $(\alpha,\beta)$.
 If $U(x_0-0)=U(\alpha)$, $U(x_0+0)=U(\beta)$, we get the conclusion.
 Suppose that $U(x_0-0)>U(\alpha)$.
 We set
\begin{gather*} 
	\beta' = \inf \left\{ x \in (x_0,\beta] \bigm| x \in C \right\}, \\
	\gamma = \sup \left\{ x \in (\alpha,x_0] \bigm| x \in C \right\}.
\end{gather*}
By Lemma \ref{LPic} and Lemma \ref{LCoi}, we observe that $C$ is a closed set.
 Thus, $\beta'>x_0$ and $\gamma<x_0$.
 Since $U(x_0-0)>U(\alpha)$, we see $\gamma>\alpha$.
 Since the jump at $x_0$ is a maximal jump, we apply Lemma \ref{LLEN} on $(\alpha,\beta')$ to get
\[
	\beta' - \alpha > C_M/\lambda,
\]
which would contradict the assumption $\beta-\alpha\le A_M/\lambda$.
 We thus conclude that $U(x_0-0)=U(\alpha)$.
 If $U(x_0+0)<U(\beta)$, we consider $-U(-x)$ instead of $U$.
 We argue in the same way.
 We apply Lemma \ref{LLEN} and get a contradiction.
 We thus conclude that $U(x_0+0)=U(\beta)$.
 The proof is now complete.
\end{proof}
\begin{proof}[Proof of Theorem \ref{TMain}]
By Lemma \ref{LPic} or more generally by Theorem \ref{TEx2}, $\inf g\leq U\leq\sup g$ on $[a,b]$ and $C$ is non-empty.
 We take an integer $m$ so that
\[
	m > (b-a)\lambda /A_M.
\]
We divide $[a,b]$ into $m$ intervals so that the length of each interval is less than $A_M/\lambda$ and the boundaries of each interval $[x_0,x_1]$ does not contain a jump point of $U$.
 (This is possible if we shift $x_0,x_1$ a little bit unless $x_0=a$, or $x_1=b$ since $J_U$ is at most a countable set.
 If $x_0=a$ (resp.\ $x_1=b$), $U$ must be continuous at $x_0=a$ ($x_1=b$) since $U$ is continuous on the coincidence set $C$ by Lemma \ref{LCoi}.) 
 If $C\cap[x_0,x_1]$ is empty or singleton, then $U$ has at most one jump by Lemma \ref{LPic}.
 We next consider the case that $C\cap[x_0,x_1]$ has at least two points.
 We set
\[
	\alpha' = \inf \left( C\cap[x_0,x_1] \right), \quad
	\beta' = \sup \left( C\cap[x_0,x_1] \right)
\]
and may assume $\alpha'<\beta'$.
 Since $\beta'-\alpha'<A_M/\lambda$, Theorem \ref{TDis} implies that $U$ only takes two values $U(x_0)$ and $U(x_1)$ and has at most one jump on $(\alpha',\beta')$.
 By Lemma \ref{LPic}, $U$ is constant on $[x_0,\alpha']$ and $[\beta',x_1]$.

%
%
 We now observe that on each $[x_0,x_1]$, $U$ has at most one jump.
 We thus conclude that $U$ is a piecewise constant function with at most $m$ jumps on $(a,b)$.
\end{proof}

\subsection{Minimizers for monotone data} \label{SSMo}

We shall prove that the bound for number of jumps is improved when $g$ is monotone.
 In other words, we shall prove Theorem \ref{TMainMon}.

We first observe the monotonicity of a minimizer for $TV_{Kg}$ when $g$ is monotone.
\begin{lem} \label{LMon}
Assume that $K$ satifies (K1) and that $g\in C[a,b]$ is non-decreasing.
 Then a minimizer of $TV_{Kg}$ (in $BV(\Omega)$) is non-decreasing.
\end{lem}
\begin{proof}
Let $U$ be a minimizer.
 We take its right continuous representation.
 Suppose that $U(x_0)>U(y_0)\geq g(y_0)$ with some $x_0<y_0$.
 Then a chopped function
\[
	v(x) = \min \left( U(y_0), U(x) \right)
\]
decreases both $TV_K$ and the fidelity term so that
\[
	TV_{Kg}(v) < TV_{Kg}(U);
\]
see Figure \ref{FCh}.
\begin{figure}[tb]
\centering 
\includegraphics[keepaspectratio, scale=0.25]{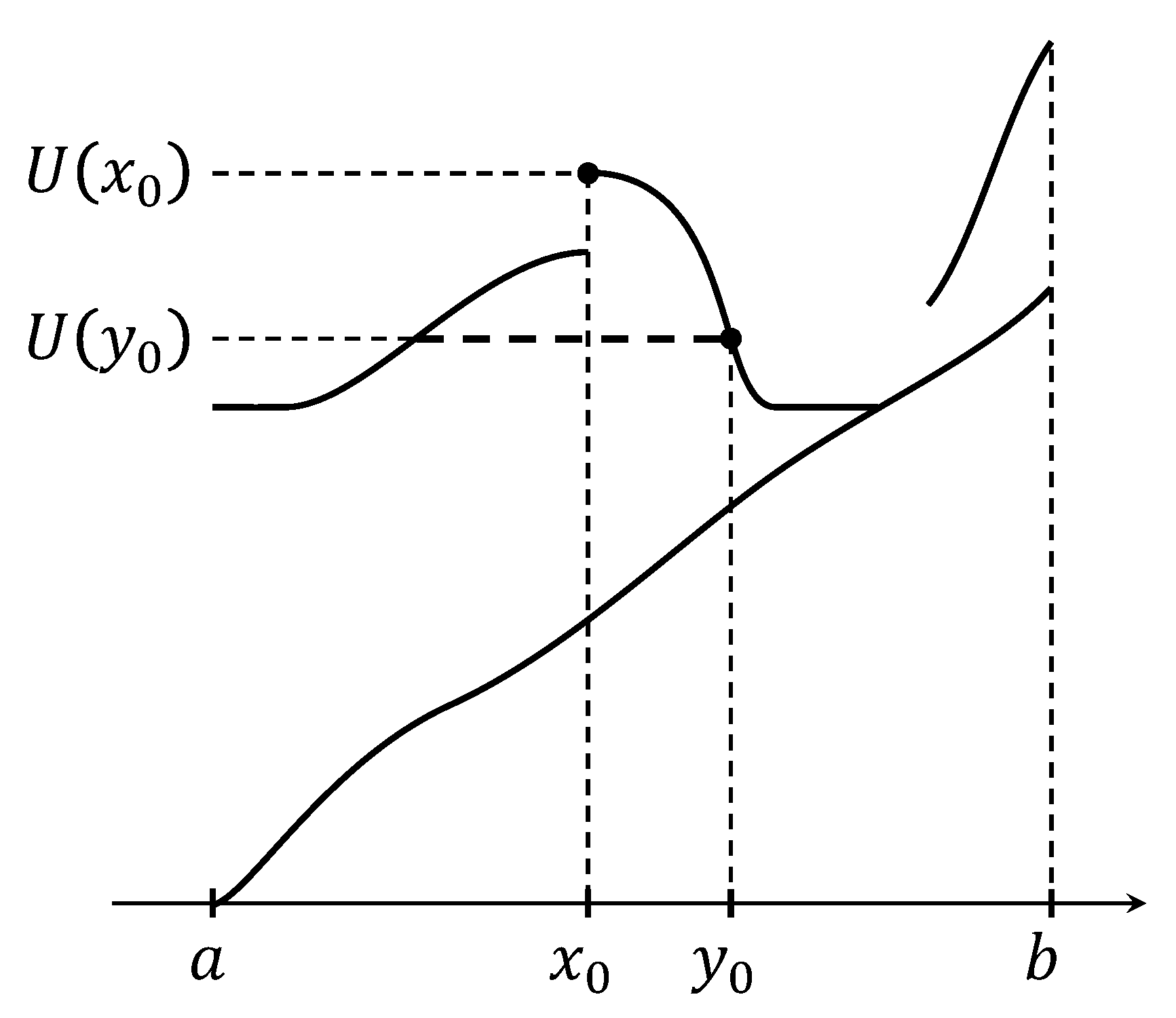}
\caption{chopped function\label{FCh}}
\end{figure}
Thus, if $U(y_0)\ge g(y_0)$, then $U(x)<U(y_0)$ for all $x<y_0$.
 A symmetric argument implies that if $U(y_0)\le g(y_0)$, then $U(x)<U(y_0)$ for all $x<y_0$.
 Thus $U$ is non-decreasing.
\end{proof}
\begin{proof}[Proof of Theorem \ref{TMainMon}]
If $g$ is monotone, a minimizer of $TV_{Kg}$ is automatically monotone by Lemma \ref{LMon}.
 We don't need to invoke Theorem \ref{TMon} so the bound $c_M/M$ is unnecessary.
 Thus Theorem \ref{TMainMon} follows from Theorem \ref{TMain}.
\end{proof}
\begin{proof}[Proof of Theorem \ref{TPie}]
If we restrict ourselves for piecewise constant minimizers, then the conclusion of Proposition \ref{PFEC}, Lemma \ref{LCGF}, Theorem \ref{TFidInc}, Lemma \ref{LLEN}, Lemma \ref{LGEs} are still valid by replacing (K3) by (K3w).
 This is because in these assertions, (K3) is used through Lemma \ref{LCpTK} but as remarked in Remark \ref{RK3w}, (K3) is unnecessary if we consider a piecewise constant minimizers.
 The conclusion of Lemma \ref{LCont} is trivial for a piecewise constant function.
 Thus, the conclusion of Theorem \ref{TDis} is still valid by replacing (K3) by (K3w) for piecewise constant minimizers.
 Since all tools including Theorem \ref{TMon} and Theorem \ref{TDis} used in the proof of Theorem \ref{TMain} and Theorem \ref{TMainMon} are proved under (K3w) instead of (K3) for piecewise constant minimizers, the proof of Theorem \ref{TPie} is now complete.
\end{proof}

It is not difficult to get a minimizer when $g$ is strictly increasing for $TV_g$, i.e.,
\[
	TV_g(u) = TV(u) + \mathcal{F}(u).
\]
By Lemma \ref{LMon}, a minimizer $U$ must be non-decreasing.
 (In this problem, $TV_g$ is strictly convex and lower semicontinuous in $L^2(\Omega)$, so there exists a unique minimizer.)
 We note that 
\[
	TV(u) = u(b) - u(a)
\]
provided that $u$ is non-decreasing.
 For positive numbers $d_1$, $d_2$, we set \\ $a_1=g^{-1}\left(g(a)+d_1\right)$, $a_2=g^{-1}\left(g(b)-d_2\right)$.
 To minimize $\mathcal{F}(u)$, we take $d_1$ and $d_2$ such that
\begin{gather*}
	d_1 = \frac{\lambda}{2} \int_a^{a_1} \left| g(a)+d_1 -g\right|^2 dx \\ 
	d_2 = \frac{\lambda}{2} \int_{a_2}^a \left| g(b)-d_2 - g\right|^2 dx.
\end{gather*}
See Figure \ref{FB}.
\begin{figure}[tb]
\centering 
\includegraphics[keepaspectratio, scale=0.25]{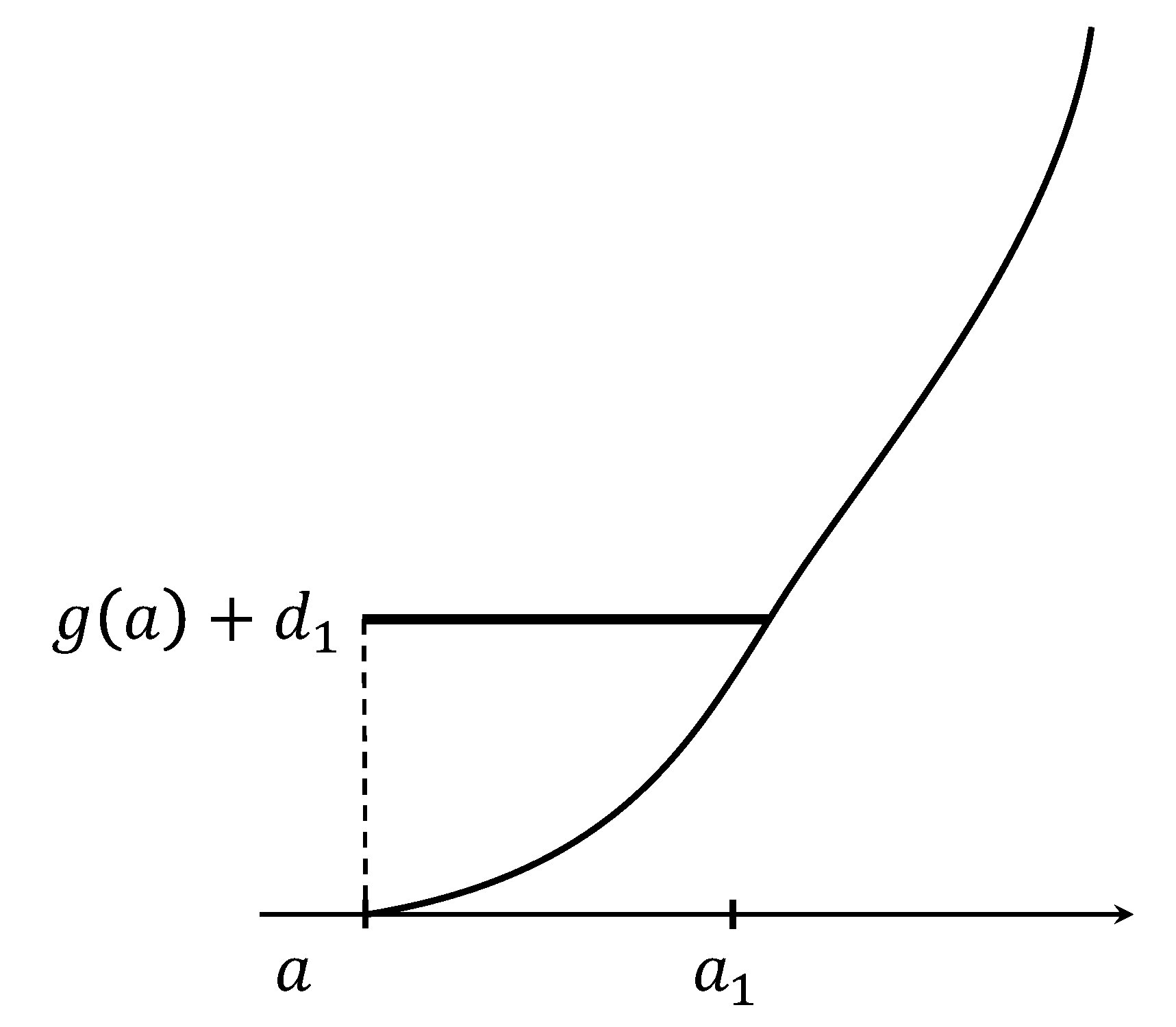}
\caption{near the boundary\label{FB}}
\end{figure}
If such $d_1$ and $d_2$ exist, then the minimizer $U$ must be
\[
	U(x) = \min \left\{ \max \left(g(a)+d_1,u(x)\right), g(b)-d_2 \right\}
\]
provided that $g(a)+d_1\leq g(b)-d_2$.
 This formation of a flat part near the boundary occurs by the natural boundary condition.
 (In general, the minimizer has no jumps if $g$ is continuous for $TV_g$ (cf.\ \cite{CL}, \cite{GKL}).)

\section{A sufficient condition for (K2)} \label{SKdef} 

We give a sufficient condition for (K2) if $K$ is derived as a limit of the Kobayashi-Warren-Carter energy, i.e., $K$ is of the form \eqref{EKdef}.

We first consider a kind of Fenchel dual of a function $f$.
 We set
\begin{equation} \label{EHdef}
	H(\rho) = \inf_{x>0} \left(\rho x + f(x) \right)
\end{equation}
for a real-valued function $f$ on $[0,\infty)$.
 If we use the Fenchel dual, it can be written as
\[
	H(\rho) = -f^*(-\rho) = - \sup_x\left((-\rho)x - \bar{f}(x) \right),
\]
where $\bar{f}(x)=f(x)$ for $x\geq0$ and $\bar{f}(x)=\infty$ for $x<0$.
 Thus, $H$ is concave in $[0,\infty)$.
 We introduce assumptions.
\begin{enumerate}
\item[(f1)] $f\in C^1(0,1]\cap C[0,1]$;
\item[(f2)] $f$ takes its minimum value $0$ at $x=1$.
 Moreover, $f>0$ on $[0,1)$ and $f\ge0$ on $\mathbb{R}$;
\item[(f3)] $f'(x)<0$ for $x\in(0,1)$ and $f'(+0)$($:=\lim_{x\downarrow0}f'(x)$)$=-\infty$.
\end{enumerate}
\begin{lem} \label{LK2}
Assume (f2).
 Then $H(\rho)>0$ for $\rho>0$ and $H(0)=0$.
 Assume further (f1) and (f3).
 Then $H(\rho)<f(0)$ for all $\rho>0$ and there is $x_\rho\in(0,1)$ for $\rho>0$ such that $H(\rho)=\rho x_\rho+f(x_\rho)$ and $H(\rho)$ is strictly increasing in $\rho$.
 Moreover, $x_\rho$ is strictly decreasing in $\rho$ and $x_\rho\uparrow 1$ as $\rho\downarrow0$.
 Furthermore, it satisfies (K2) with $K=H$ if $f$ satisfies
\begin{equation} \label{EHC}
	\varliminf_{\rho\downarrow0} \left(f\circ(-f')^{-1}\right) (\rho) \bigm/ \rho^2 > 0.
\end{equation}
Here we define
\[
	(-f')^{-1}(\rho) = \min \left\{ x \in [0,1] \bigm|
	f'(x) = -\rho \right\}.
\] 
\end{lem}
\begin{proof}
The positivity for $H(\rho)$ for $\rho>0$ and $H(0)=0$ is clear by the definition \eqref{EHdef}.
 Also, the existence in $[0,1]$ of a minimizer easily follows by (f1) and (f2).
 Although $x_\rho$ may not be unique, $x_\rho\uparrow1$ as well as monotonicity of $x_\rho$ is guaranteed by (f3). 
 The assumption $f'(+0)=-\infty$ is invoked so that $x_\rho>0$ for all $\rho>0$. The strict monotonicity of $x_\rho$ yields the strict monotonicity of $H$.
 The bound $H(\rho)<f(0)$ is rather clear.

It remains to prove that \eqref{EHC} yields (K2).
 Since $x_\rho$ is the minimizer, it must satisfy
\[
	\rho = (-f') (x_\rho).
\]
We take $x_\rho=(-f')^{-1}(\rho)$ in Lemma \ref{LK2}.
 Since
\[
	H(\rho) = \rho x_\rho + f(x_\rho)
\]
by definition, we observe that for $\delta\in(0,1)$
\begin{align}
\begin{aligned} \label{EEH} 
	H(\delta\rho) - \delta H(\rho) &\geq \rho_1 x_{\rho_1} + f(x_{\rho_1})
	- \left( \delta\rho x_{\rho_1} + \delta f(x_{\rho_1}) \right)\ \text{with}\ \rho_1=\delta\rho \\
	&= (1-\delta) f(x_{\rho_1}). 
\end{aligned}
\end{align}
For $\rho_2=(1-\delta)\rho$, we have
\[
	H \left( (1-\delta)\rho \right) - (1-\delta) H (\rho) \geq \delta f(x_{\rho_2}).
\]
We thus observe that
\[
	H(\rho_1) + H(\rho_2) - H(\rho)
	\geq (1-\delta) f(x_{\rho_1}) + \delta f(x_{\rho_2}).
\]
By \eqref{EHC}, we may assume that
\[
	\left( f \circ (-f')^{-1} \right) (\rho) \geq C_M \rho^2
\]
with some $C_M>0$ provided that $0\leq\rho\leq M$.
Thus
\begin{align*}
	(1-\delta) f(x_{\rho_1}) &= (1-\delta) \left( f \circ (-f')^{-1} \right) (\rho_1) \geq C_M(1-\delta)(\delta\rho)^2, \\
	\delta f(x_{\rho_2}) &\geq C_M \delta \left( (1-\delta) \rho \right)^2.  
\end{align*}
We now observe that
\begin{align*}
	H(\rho_1) + H(\rho_2) - H(\rho) &\geq C_M \left( (1-\delta)\delta\rho^2 \delta + (1-\delta)\delta\rho^2(1-\delta) \right) \\
	&= C_M \left( (1-\delta) \delta\rho^2 \right) 
	= C_M \rho_1 \rho_2.  
\end{align*}
We have proved (K2) for $H$.
\end{proof}
\begin{lem} \label{LK3}
Assume that (f1), (f2) and (f3).
 Then $\lim_{\rho\downarrow0}H(\rho)/\rho=1$.
 In other words, $H$ satisfies (K3) with $K=H$.
\end{lem}
\begin{proof}
Taking $x=1$ in \eqref{EHdef}, we see that $H(\rho)\leq\rho$.
 We observe that
\begin{align*}
	H(\rho) - \rho &= \min_{x>0} \left( \rho(x-1) + f(x) \right) \\
	&= \rho \left( (x_\rho-1) + f(x_\rho)\bigm/\rho \right)\ \text{or} \\
	\frac{H(\rho)}{\rho} -1 &= x_\rho -1 + \frac{f(x_\rho)}{\rho}. 
\end{align*}
Since $f\geq0$ and $x_\rho\to1$ as $\rho\downarrow0$, we conclude
\[
	\varliminf_{\rho\downarrow0} \left( \frac{H(\rho)}{\rho} -1 \right) \geq 0+0,
\]
which now yields (K3).
\end{proof}
\begin{remark} \label{RKCon}
\begin{enumerate}
\item[(i)] Without \eqref{EHC} we only get (K2w) since $f\geq0$ and the estimate \eqref{EEH} yields subadditivity.
\item[(i\hspace{-1pt}i)] If $f(x)=|x-1|^m$ for $m>0$, then \eqref{EHC} holds if and only if $m\geq2$.
 In fact, $f'(x)=m|x-1|^{m-2}(x-1)$ so that $(-f')^{-1}(\rho)=1-(\rho/m)^{1/(m-1)}$.
 The function $\left(f\circ(-f')^{-1}\right)(\rho)=(\rho/m)^{m/(m-1)}$ so \eqref{EHC} holds if and only if $m\geq2$.
\item[(i\hspace{-1pt}i\hspace{-1pt}i)] We may take other element of a preimage of $-f'$ of $\rho$ but as a sufficient condition the present choice is the weakest assumption.
\end{enumerate}
\end{remark}
We come back to \eqref{EKdef}.
 In other words,
\[
	K(\rho) = \min_\xi \left(\rho(\xi_+)^2 + 2G(\xi)\right), \quad
	G(\xi) = \left| \int_1^\xi \sqrt{F(\tau)}\; d\tau \right|.
\]
We assume that
\begin{enumerate}
\item[(F1)] $F\in C[0,\infty)$ and $F(x)$ takes the only minimum $0$ at $x=1$.
\end{enumerate}
If we set $f(x)=2G(x^{1/2})$, the property (F1) implies (f1), (f2) and (f3).
 Indeed, (f1), (f2) as well as $f'<0$ on $(0,1)$ are easy to check.
 Since 
\[
	f'(x) = 2G'(x^{1/2}) \frac12 x^{-1/2}
\]
and $G'(0)<0$, we see that $f'(+0)=-\infty$.
 Since \eqref{EHC} is property near $x=1$, \eqref{EHC} for $f$ and $G$ are equivalent.
 We thus obtain an sufficient condition so that $K$ in \eqref{EKdef} satisfies (K1), (K2) and (K3)
\begin{prop} \label{PSuff}
Assume (F1) and
\begin{equation} \label{EKSu}
\varliminf_{\rho\downarrow0} \left( G \circ (-G')^{-1} \right) (\rho) \bigm/ \rho^2 > 0.
\end{equation}
Then $K$ in \eqref{EKdef} satisfies (K1), (K2) and (K3).
\end{prop}
\begin{proof}
By Lemma \ref{LK2}, (K2) is fulfilled.
 Since $K$ is concave in $[0,\infty)$, it is clear that $K$ is continuous on $[0,\infty)$.
 Lemma \ref{LK2} shows that $K$ is strictly increasing so we have proved a property stronger than (K1), continuity and strictly increasing.
 The property (K3) follows from Lemma \ref{LK3}.
\end{proof}

We conclude this section by examining the property \eqref{EKSu}.
 This condition is equivalent to saying that
\[
	\varliminf_{\rho\downarrow0} \int_{F^{-1}(\rho^2)}^1 \sqrt{F(\tau)}\; d\tau \bigm/ \rho^2 > 0,
\]
where $F^{-1}(\rho^2)=\min\left\{x\in[0,1] \bigm| F(x)=\rho^2 \right\}$.
 We set
\[
	\bar{F}(x) = F(1-x).
\]
The condition \eqref{EKSu} is now equivalent to
\begin{equation} \label{EFban}
	\varliminf_{\rho\downarrow0} \int_0^{\bar{F}^{-1}(\rho^2)} \sqrt{\bar{F}(\tau)} \; d\tau \bigm/ \rho^2 > 0,
\end{equation}
where $\bar{F}^{-1}(y)=\max\left\{x\in(0,1)\bigm|\bar{F}(x)=y\right\}$.
 To simplify the argument, we further assume that
\begin{enumerate}
\item[(F2)] $F'<0$ in $(0,1)$ so that the inverse function $F^{-1}$ in $\left(0,F(0)\right)$ is uniquely determined.
\end{enumerate}
(Note that $f(x)=2G(x^{1/2})$ is now convex in $(0,1)$ since
\[
	\frac{d^2}{dx^2} G(x^{1/2}) = \frac{d}{dx} \frac{G'}{2x^{1/2}}
	= \frac{G''}{2x^{1/2}} \cdot \frac{1}{2x^{1/2}}
	- G' \frac12 \cdot \frac12 x^{-3/2}
\]
and $G'<0$ and $G''>0$ on $(0,1)$.)
 If we assume (F1) and (F2), by changing the variable of integration $\tau=\bar{F}^{-1}(s^2)$, we have
\[
	\int_0^{\bar{F}^{-1}(\rho^2)} \sqrt{\bar{F}(\tau)} \; d\tau
	= \int_0^{\rho^2} s \frac{d\tau}{ds}\; ds
	= \int_0^{\rho^2} 2s^2 (\bar{F}^{-1})' (s^2)\; ds.
\]
The condition \eqref{EFban} is fulfilled if 
\[
	\varliminf_{\sigma\downarrow0} \sigma (\bar{F}^{-1})' (\sigma) > 0
\]
or equivalently
\[
	\varlimsup_{\eta\downarrow0} \bar{F}'(\eta)\Bigm/ \eta <\infty.
\]
We thus obtain a simple sufficient condition.
\begin{thm} \label{TSuff}
Assume that (F1) and (F2).
 Then $K$ in \eqref{EKdef} satisfies (K1), (K2), (K3) provided that
\[
	\varlimsup_{x\uparrow1} F'(x) / (x-1) <\infty.
\]
\end{thm}


\section*{Acknowledgements}
The work of the first author was partly supported by JSPS KAKENHI Grant Numbers JP19H00639, JP20K20342, JP24K00531 and JP24H00183 and by Arithmer Inc., Daikin Industries, Ltd.\ and Ebara Corporation through collaborative grants.
 The work of the third author was partly supported by JSPS KAKENHI Grant Number JP20K20342.
 The work of the fifth author was partly supported by JSPS KAKENHI Grant Numbers JP22K03425, JP22K18677, 23H00086.


\end{document}